\documentclass{amsart}
\usepackage[margin=1.32in]{geometry}
\usepackage{mathabx}
\usepackage{amsfonts}
\usepackage{lineno,hyperref}
\usepackage{amsmath}
\usepackage{amssymb}
\usepackage{amsthm}
\usepackage{cancel}
\usepackage{tikz-cd}
\usepackage{arydshln}
\usepackage{graphicx}
\usepackage{cite}

\newcommand{\RR}{\mathbb{R}}
\newcommand{\CC}{\mathbb{C}}
\newcommand{\NN}{\mathbb{N}}

\newcommand{\ZZ}{\mathbb{Z}}
\newcommand{\QQ}{\mathbb{Q}}
\newcommand{\OO}{\mathcal{O}}
\newcommand{\cP}{\mathcal{P}}
\newcommand{\HH}{\mathbb{H}}

\newcommand{\nrm}{\text{nrm}}

\newcommand{\tr}{\text{tr}}
\newcommand{\disc}{\text{disc}}
\newcommand{\discrd}{\text{discrd}}

\newcommand{\Mat}{\text{Mat}}
\newcommand{\Isom}{\text{Isom}}
\newcommand{\Mob}{\text{M\"ob}}
\newcommand{\BMS}{\mathrm{BMS}}
\newcommand{\sk}{\mathrm{sk}}
\newcommand{\PS}{\mathrm{PS}}

\theoremstyle{plain}
\newtheorem{theorem}{Theorem}[section]
\newtheorem{definition}[theorem]{Definition}
\newtheorem{corollary}[theorem]{Corollary}
\newtheorem{conjecture}{Conjecture}[section]
\newtheorem{lemma}[theorem]{Lemma}

\theoremstyle{remark}

\newtheorem{remark}{Remark}[section]

\begin{document}
\title{Asymptotic Density of Apollonian-Type Packings}
\author{Matthew Litman}
\address{Department of Mathematics, UC Davis, One Shields Ave, Davis, CA 95616}
\email{mclitman@ucdavis.edu}

\author{Arseniy (Senia) Sheydvasser}
\address{Department of Mathematics, Technion, Haifa}
\email{sheydvasser@gmail.com}

\date{\today}

\begin{abstract}
We consider two seemingly unconnected problems: first, under what circumstances are arithmetic groups like $SL(2,\OO)$ generated by elementary matrices; second, when do certain classes of circle/sphere packings fill up space? We show that these are in fact deeply related, leading to some new results, new proofs of old results, and interesting conjectures.
\end{abstract}

\maketitle

\section{Introduction:}\label{section: introduction}

There has been a lot of interest in studying ``Apollonian-like" circle and sphere packings in recent years, with many different constructions and definitions\cite{GuettlerMallows2010,Stange2015,KontorovichNakamura2019,kapovich_kontorovich_2021,Sheydvasser2019,Dias_2014,Nakamura_2014}. Although there isn't any unified agreement on what ``Apollonian-like" should mean, at minimum it should be a set of oriented $(n - 2)$-spheres with non-intersecting interiors such that there exists a non-trivial subgroup $G$ of $O^+(n,1)$ acting on this set. For the most part, studying such sets has concentrated on taking machinery from number theory, analysis, and geometry to better understand such constructions; our goal in this paper is in some sense to do the reverse. Specifically, we wish to connect the problem of elementary generation of special linear groups to the problem of studying Apollonian type sphere packings. Let $K$ be an algebraic number field, $\mathfrak{o}_K$ its ring of integers, $SL(n,\mathfrak{o}_K)$ the special linear group on $\mathfrak{o}_K^n$, and $E(n,\mathfrak{o}_K)$ the subgroup generated by upper and lower triangular matrices. Then,
    \begin{itemize}
        \item if $n > 2$, then $SL(n,\mathfrak{o}_K) = E(n,\mathfrak{o}_K)$\cite{Bass_Milnor_Serre_1967},
        \item if $K$ is not an imaginary quadratic field, then $SL(n,\mathfrak{o}_K) = E(n, \mathfrak{o}_K)$\cite{Vaserstein_1972},
        \item if $K$ is an imaginary quadratic field, then $SL(2,\mathfrak{o}_K) = E(2,\mathfrak{o}_K)$ if and only if $\mathfrak{o}_K$ is a Euclidean ring\cite{Cohn1966}. Furthermore, if $\mathfrak{o}_K$ is not Euclidean, then $E(2,\mathfrak{o}_K)$ is an infinite-index, non-normal subgroup\cite{Nica_2011}.
    \end{itemize}
    
\noindent This problem can be related to circle packings as follows. If $K$ is an imaginary quadratic field, define its corresponding Schmidt arrangement $\mathcal{S}_K$ as the orbit of the real line under the action of $SL(2,\mathfrak{o}_K)$ on $\CC \cup \{\infty\}$ via M\"obius transformations. Then $\mathcal{S}_K$ is connected as a set if and only if $SL(2,\mathfrak{o}_K) = E(2,\mathfrak{o}_K)$\cite{Stange2018}. This Schmidt arrangement is not Apollonian-like because each circle in it contains infinitely many other circles in it. However, there exists a naturally-defined Apollonian type circle packing $\mathcal{A}_K$ inside of it, dubbed the $K$-Apollonian packing (containing the real line) by Katherine Stange\cite{Stange2015}; this has many nice properties, including the fact that its topological closure is the limit set of a thin, geometrically finite Kleinian group. We can define the \emph{density} $\delta(\mathcal{A}_K)$ as roughly the proportion of the total available space that it takes up (a more precise definition will be given in Section \ref{section: apollonian type}), and this allows us to state a curious geometric connection.

\begin{theorem}\label{thm: unreasonable slightness for circle packings}
Let $K$ be an imaginary quadratic field and $\mathfrak{o}_K$ its ring of integers. Exactly one of the following happens.
    \begin{enumerate}
        \item $\mathfrak{o}_K$ is norm-Euclidean, $\delta\left(\mathcal{A}_K\right) = 1$, and $SL(2,\mathfrak{o}_K) = E(2,\mathfrak{o}_K)$.
        \item $\mathfrak{o}_K$ is not Euclidean,  $\delta\left(\mathcal{A}_K\right) < 1$, and $E(2,\mathfrak{o}_K)$ is an infinite-index, non-normal subgroup of $SL(2,\mathfrak{o}_K)$.
    \end{enumerate}
    
\noindent Furthermore, if $D$ is the discriminant of $K$, then $\delta(\mathcal{A}_K) \rightarrow 0$ as $D \rightarrow \infty$.
\end{theorem}

This theorem will be proved in Section \ref{section: forbidden balls}---it is mostly an immediate corollary of previously known results, although we do note that it gives a new, geometric proof of the fact that there are only finitely many imaginary quadratic fields $K$ such that $SL(2,\mathfrak{o}_K)$ is generated by elementary matrices. More interestingly, we conjecture that this connection can be generalized in at least two other settings.

\begin{conjecture}[Unreasonable Slightness]\label{conjecture: unreasonable slightness}
Let $dim = 3$, $4$, or $5$. Let $R(\QQ)$ be either an imaginary quadratic field (if $dim = 3$), a rational, definite quaternion algebra equipped with an orthogonal involution $\ddagger$ (if $dim = 4$), or a rational, definite quaternion algebra (if $dim = 5$). Let $R(\ZZ)$ be a maximal order (or $\ddagger$-order) of $R(\QQ)$. Let $G(\ZZ)$ either be $SL(2,R(\ZZ))$ if $dim = 3,5$ or $SL^\ddagger(2,R(\ZZ))$ if $dim = 4$. Let $E(\ZZ)$ be the subgroup generated by upper and lower triangular matrices. Then there exists a corresponding Apollonian type $(dim - 2)$-sphere packing $\mathcal{A}_{R(\ZZ),u}$ whose closure is the limit set of a discrete, thin, geometrically finite subgroup of the isometry group of hyperbolic $dim$-space $\HH^{dim}$, such that exactly one of the following is true.
    \begin{enumerate}
        \item $R(\ZZ)$ is norm-Euclidean (or norm $\ddagger$-Euclidean, if $dim = 4$), $\delta\left(\mathcal{A}_{R(\ZZ),u}\right) = 1$, and $G(\ZZ) = E(\ZZ)$.
        \item $R(\ZZ)$ is not Euclidean (or $\ddagger$-Euclidean, if $dim = 4)$,  $\delta\left(\mathcal{A}_{R(\ZZ),u}\right) < 1$, and $E(\ZZ)$ is an infinite-index, non-normal subgroup of $G(\ZZ)$.
    \end{enumerate}
\end{conjecture}

\begin{remark}
The name ``unreasonable slightness" comes from the paper \emph{The Unreasonable Slightness of $E_2$ Over Imaginary Quadratic Rings} by Bogdan Nica\cite{Nica_2011}.
\end{remark}

\begin{figure}
    \centering
    \begin{tabular}{cc}
    \includegraphics[height = 0.2\textheight]{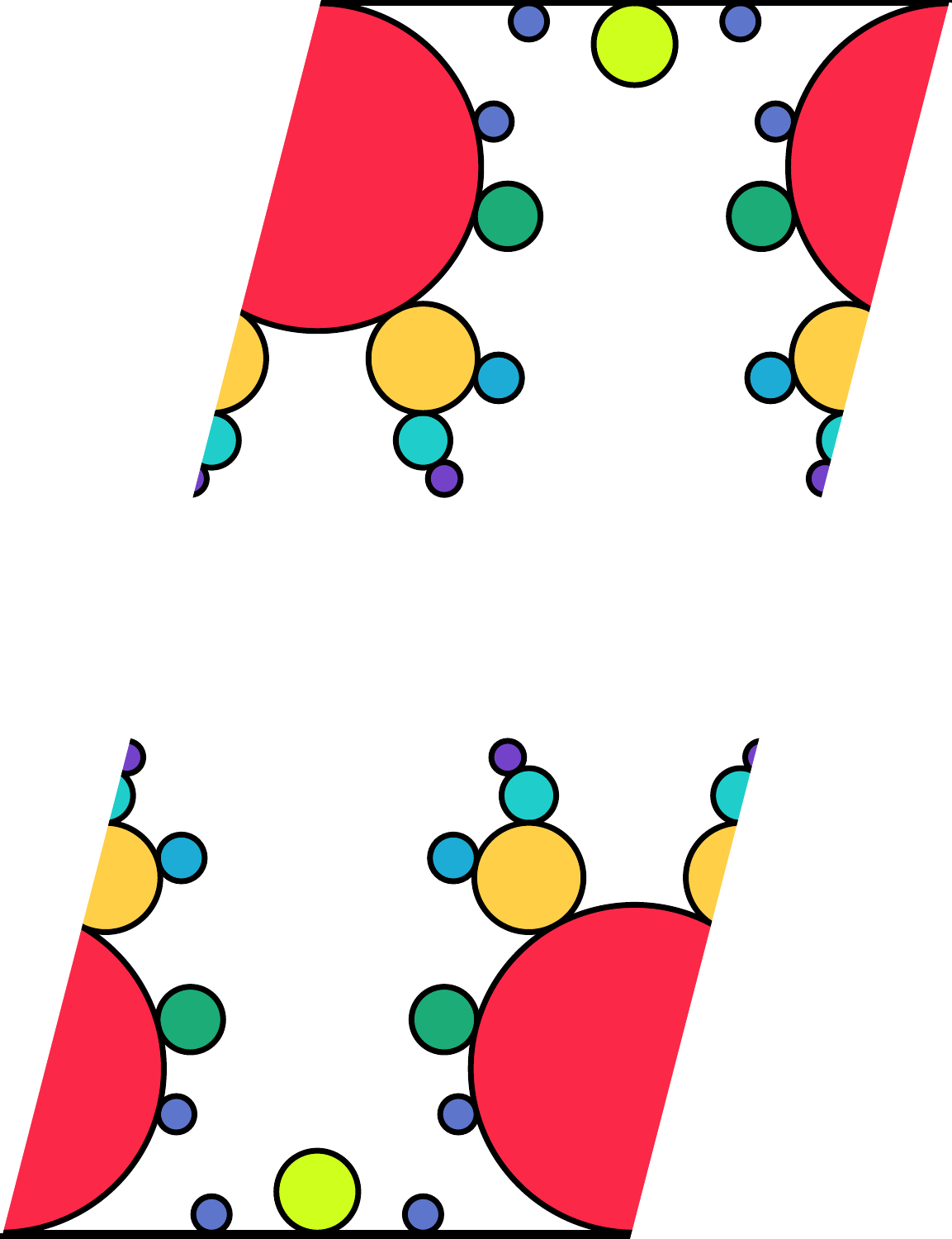} & \includegraphics[height = 0.2\textheight]{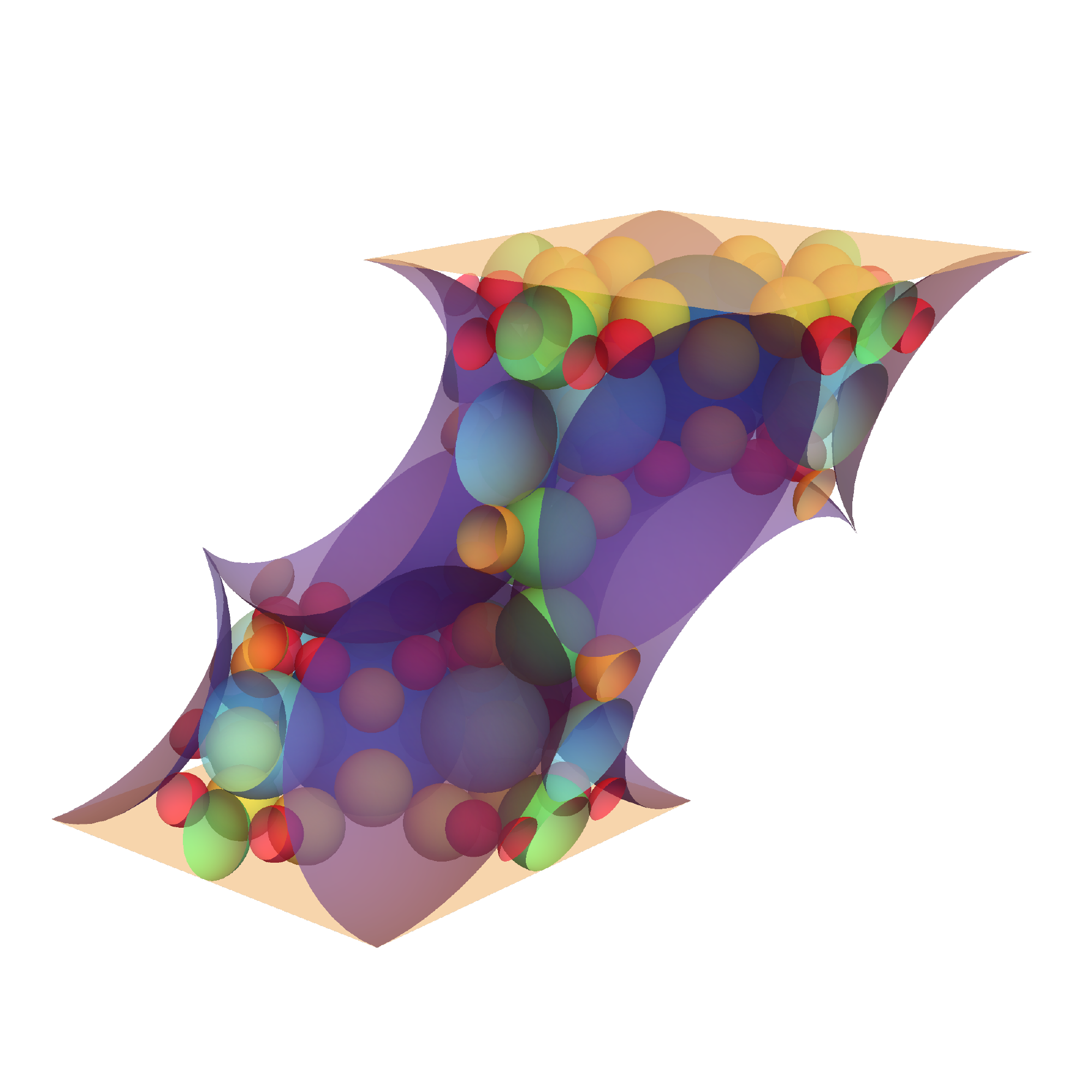}
    \end{tabular}
    \caption{Packings of Apollonian type corresponding to $\ZZ\left[\frac{1 + \sqrt{-5}}{2}\right]$ and $\ZZ \oplus \ZZ i \oplus \ZZ \frac{1 + i + j}{2} \oplus \ZZ \frac{j + ij}{2}$ in $\left(\frac{-1,-6}{\QQ}\right)$. The first one has density less than $1$; the second has density $1$.}
    \label{fig:apollonian_packings}
\end{figure}

\begin{conjecture}\label{conjecture: asymptotic diminishing of density}
With $R(\ZZ)$, $\mathcal{A}_{R(\ZZ),u}$ defined as above, $\delta(\mathcal{A}_{R(\ZZ),u}) \rightarrow 0$ as $D \rightarrow \infty$, where $D$ is the discriminant of $R(\ZZ)$.
\end{conjecture}

Notice that Theorem \ref{thm: unreasonable slightness for circle packings} can be summarized as just saying that Conjecture \ref{conjecture: unreasonable slightness} and Conjecture \ref{conjecture: asymptotic diminishing of density} are true for $dim = 3$. It might be a little less obvious what the conjecture is saying for the other two cases, for two reasons: first, the reader may be unfamiliar with the notions of $\ddagger$-rings, $\ddagger$-Euclidean rings, and groups $SL^\ddagger(2,R(\ZZ))$. We will give precise definitions in Section \ref{section: basic definitions}, although the reader should view these as analogs of imaginary quadratic rings, Euclidean rings, and the Bianchi groups, respectively. Secondly, we have not given a definition of what the Apollonian type packing $\mathcal{A}_{R(\ZZ),u}$ ought to be for $dim \neq 3$. In fact, for reasons discussed in Section \ref{section: apollonian type}, we do not know what should be the correct definition in general, and we have to make the following restriction.

\begin{definition}
Let $u \in \RR^{dim - 1}$ be a non-zero vector with first component equal to $0$. Let $S_u$ be the oriented hyper-plane in $\RR^{dim - 1}$ passing through the origin, with normal vector $u$. We shall say that $u$ is a \emph{covering vector} if open unit balls centered on points in $S_u \cap R(\ZZ)$ cover $S_u$.
\end{definition}

\begin{remark}
We will eventually be identifying $\RR^{dim - 1}$ with a (subset of a) division algebra---the requirement that the first component be equal to $0$ is equivalent to saying that $1 \in S_u$.
\end{remark}

If $R(\ZZ)$ has a covering vector $u$, we shall show in Section \ref{section: apollonian type} that it is possible to define a corresponding Apollonian type packing $\mathcal{A}_{R(\ZZ),u}$, which we shall show is the limit set of a discrete, thin, geometrically finite group. Our main result can then be phrased as follows.

\begin{theorem}\label{thm: main theorem}
If $R(\ZZ)$ has a covering vector $u$ and its discriminant $D$ is sufficiently large, then $E(\ZZ)$ is an infinite-index subgroup of $G(\ZZ)$ and $\delta(\mathcal{A}_{R(\ZZ),u}) < 1$. Furthermore, if we only index over all such orders $R(\ZZ)$, then $\delta(\mathcal{A}_{R(\ZZ),u}) \rightarrow 0$ as $D \rightarrow \infty$.
\end{theorem}

\begin{remark}
In fact, we show that $|\disc(R(\ZZ))| > 33$ is sufficient---this can be further improved, although it depends on the dimension $dim$ and the class of $R(\ZZ)$.
\end{remark}

While some special cases of this theorem were proved in a paper of the second author\cite{Sheydvasser2019}, this is a new result---in particular, we prove that $E(\ZZ)$ is an infinite index subgroup for a variety of arithmetic groups $G(\ZZ)$ for which this was not previously known. The key ingredients in the proof are geometric in nature, coming from the tangency structure of $\mathcal{A}_{R(\ZZ),u}$.

We also investigate the problem of studying the asymptotic behaviour of $\delta(\mathcal{A}_{R(\ZZ),u})$ as a function of the discriminant---we believe that this might be of independent interest. In Section \ref{section: density estimates} we provide an asymptotic growth formula for $\delta(\mathcal{A}_{R(\ZZ),u})$ for appropriate choices of $R(\ZZ)$ and $u$. This formula shows that, assuming some natural conjectures on the growth rates of various geometric constants, $\delta(\mathcal{A}_{R(\ZZ),u})$ grows like $\sqrt{|\disc(R(\ZZ))|}^{-1}$, and in conjunction with Theorem \ref{thm: unreasonable slightness for circle packings} tells us that $\mathfrak{o}_K$ is not Euclidean and $E(2,\mathfrak{o}_K)$ does not generate $SL(2,\mathfrak{o}_K)$ as the discriminant get large. The constants present in the asymptotic density formula depend on the spectral theory of the symmetry group of the packing and gaining control on these could be of interest to a computational dynamicist. In Section \ref{section: forbidden balls}, we give a heuristic argument for the behaviour of the density as the discriminant grows by introducing the notion of \textit{forbidden balls}, a generalization of Stange's ghost circles. This approach provides us with an explicit upper bound on the density of the packing for dimension $3$, and the approach can be iterated to obtain finer and finer estimates of the density. For dimensions $4$ and $5$, we provide an exhaustive list of orders in which forbidden balls appear along with the coordinates of forbidden ball for each case.

\section{Basic Definitions:}\label{section: basic definitions}

We start by giving definitions for the rings $R(\ZZ)$ and associated objects. Throughout, we shall use the notation that
    \begin{align*}
        H = \left(\frac{a,b}{\mathbb{F}}\right)
    \end{align*}
    
\noindent denotes the quaternion algebra over a field $\mathbb{F}$ generated by elements $i,j$ with the defining relations $i^2 = a \in \mathbb{F}^\times$, $j^2 = b \in \mathbb{F}^\times$, and $ij = -ji$; we assume throughout that the characteristic of $\mathbb{F}$ is not $2$. Such a quaternion algebra is called rational if $\mathbb{F} = \RR$; it is called definite if additionally $a < 0$ and $b < 0$, or equivalently this algebra can be viewed as a sub-algebra of the classical Hamilton quaternions $H_\RR$. If $R$ is a Dedekind ring with field of fractions $\mathbb{F}$ and $H$ is a finite-dimensional algebra over $\mathbb{F}$, then an ($R-$) \emph{order} $\OO$ of $H$ is an $R$-lattice (that is, a finitely generated sub-module of $H$ such that $\mathbb{F}\OO = H$) which is also a sub-ring of $H$. An order is called \emph{maximal} if it is not contained in any larger order. We will primarily look at the case where $R = \ZZ$ and $H$ is either an imaginary quadratic field (if $dim = 3$) or a rational, definite quaternion algebra (if $dim = 4,5$). In either case, there is a natural conjugation map $\alpha \mapsto \overline{\alpha}$, which we use to define the \emph{trace} $\tr(\alpha) = \alpha + \overline{\alpha}$ and the \emph{norm} $\nrm(\alpha) = \alpha\overline{\alpha}$. If $S$ is a subset of an algebra with such a conjugation map, we shall write $S_0$ to denote the trace zero subset of it.

Given a ring $R$, an \emph{involution} on $R$ is a group homomorphism $\sigma: R \rightarrow R$ such that $\sigma^2 = id$ and $\sigma(xy) = \sigma(y)\sigma(x)$ for all $x,y \in R$. A homomorphism of rings with involutions $\varphi:(R,\sigma) \rightarrow (S,\sigma')$ is a ring homomorphism such that $\varphi \circ \sigma = \sigma' \circ \varphi$. Throughout, we shall write $R^+$ to denote the subset of $R$ consisting of elements $\alpha \in R$ such that $\sigma(\alpha) = \alpha$; $R^-$ will denote the subset of $\alpha \in R$ such that $\sigma(\alpha) = -\alpha$. Quaternion algebras over $\mathbb{F}$ admit exactly two kinds of $\mathbb{F}$-linear involutions: the standard involution
    \begin{align*}
        \overline{x + yi + zj + tij} = x - yi - zj -tij,
    \end{align*}
    
\noindent and the orthogonal involutions which act as $id$ on a subspace of dimension $3$ and as $-id$ on the orthogonal subspace of dimension $1$. Given an orthogonal involution $\ddagger$ on a quaternion algebra $H$, one can always choose a basis $1,i,j,ij$ for $H$ such that one can express this orthogonal involution as
    \begin{align*}
        \left(x + yi + zj + tij\right)^\ddagger = x + yi + zj - tij,
    \end{align*}
    
\noindent which is the convention that we shall take henceforth. A $\ddagger$-\emph{order} of $(H,\ddagger)$ is an order $\OO$ such that $\OO^\ddagger = \OO$---equivalently, this is a sub-ring with involution which is also an order. A $\ddagger$-order is called \emph{maximal} if it is not contained in any larger $\ddagger$-order. Such objects were originally studied by Scharlau in the 1970s\cite{Scharlau1974}, but it was only fairly recently that it was shown that there is an efficient method for determining whether a $\ddagger$-order of quaternion algebra over a global field is maximal, as we shall explain shortly.

There are four notions of discriminant that we shall have to consider. If $R$ is a Dedekind domain with field of fractions $\mathbb{F}$, $H$ is a finite-dimensional $\mathbb{F}$-algebra, and $\OO$ is an order of $H$, then the \emph{discriminant} of $\OO$ is the ideal $\disc(\OO)$ of $R$ generated by
    \begin{align*}
        \det\left(\tr(e_i \overline{e_j})\right)
    \end{align*}
    
\noindent where $e_1, e_2,\ldots$ is an $R$-basis for $\OO$. If $\OO = \mathfrak{o}_K$, the ring of integers of an imaginary quadratic field $K = \QQ(\sqrt{n})$ for some square-free integer $n$, then
    \begin{align*}
        \disc(\mathfrak{o}_K) = \begin{cases} n & \text{if $n = 1 \mod 4$} \\ 4n & \text{otherwise.} \end{cases}
    \end{align*}
    
\noindent If $\OO$ is an order of a quaternion algebra, then $\disc(\OO)$ is always a square ideal. In either case, we will write $\discrd(\OO)$ to mean the \emph{reduced discriminant}, which is the square-free part of the ideal $\disc(\OO)$. Closely related is the notion of the \emph{discriminant} of a quaternion algebra $H$ over $\mathbb{F}$, which is the ideal
    \begin{align*}
        \disc(H) = \prod_{\substack{\mathfrak{p} \subset R \text{ prime ideal} \\ H_\mathfrak{p} \text{ a division algebra}}} \mathfrak{p}.
    \end{align*}
    
\noindent An order $\OO$ of $H$ is maximal if and only if $\discrd(\OO) = \disc(H)$. Finally, the \emph{discriminant} of an orthogonal involution $\ddagger$ on a quaternion algebra $H$ over $\mathbb{F}$ is the set $\disc(\ddagger) = \alpha^2 \left(\mathbb{F}^{\times}\right)^2$, where $\alpha$ is any non-zero element in $H^-$. Unlike, the others, this is not an ideal, but this is easily fixed---we write $\iota(\disc(\ddagger))$ for the $R$-ideal generated by $\disc(\ddagger) \cap R$. With these conventions, we get an easy way to characterize maximal $\ddagger$-orders: if $\mathbb{K}$ is a global or local field (such as an algebraic number field or one of its localizations), $\mathfrak{o}_{\mathbb{K}}$ is its ring of integers, $H$ is a quaternion algebra over $\mathbb{K}$, and $\OO$ is a $\ddagger$-order of $H$, then $\OO$ is maximal if and only if $\discrd(\OO) = \disc(H) \cap \iota(\disc(\ddagger))$\cite{Sheydvasser2017}. If $I$ is a $\ZZ$-ideal, we shall write $|I|$ to denote its smallest non-negative generator.

There are also some associated groups that we need to define. If $(R,\sigma)$ is a ring with involution, then we can extend this to a ring with involution $(\Mat(2,R),\hat{\sigma})$, where
    \begin{align*}
        \hat{\sigma}\left(\begin{pmatrix} a & b \\ c & d \end{pmatrix}\right) = \begin{pmatrix} \sigma(d) & -\sigma(b) \\ -\sigma(c) & \sigma(a) \end{pmatrix}.
    \end{align*}
    
\noindent Then we may define the twisted special linear group
    \begin{align*}
        SL^\sigma(2,R) = \left\{M \in \Mat(2,R)\middle| \hat{\sigma}(M)M = M\hat{\sigma}(M) = 1\right\}.
    \end{align*}
    
\noindent If $R$ is a commutative ring and $\sigma$ is the identity, then $SL^\sigma(2,R) = SL(2,R)$, motivating the name. We shall be primarily interested in the case where $R$ is a subring with involution of a definite, rational quaternion algebra $H$ with orthogonal involution $\sigma = \ddagger$; in this context, these groups were originally introduced in a paper of the second author\cite{Sheydvasser2019}. It is readily checked that $SL^\ddagger(2,H)$ is a subgroup of $SL(2,H)$. As $H$ is a noncommutative division algebra, it might not be obvious what is meant by the special linear group in this case---there are a number of ways to define it, such as by introducing the Dieudonn\'e determinant, but we will use the following simple characterization: $H$ can be embedded as a ring into $\Mat(2,\CC)$, and therefore $\Mat(2,H)$ can be viewed as a subring of $\Mat(4,\CC)$. With this in mind, we can define $SL(2,H) = \Mat(2,H) \cap SL(4,\CC)$.

Finally, we need to define Euclidean and $\ddagger$-Euclidean ring. Recall that a subring $R$ of a division algebra is \emph{Euclidean} if there exists a function $\Phi:R \rightarrow W$ to some well-ordered set $W$ such that for all $a,b \in R$ with $b \neq 0$, there exists $q \in R$ such that $\Phi(a - bq) < \Phi(b)$. In this case, the usual Euclidean algorithm can be applied to $R$; a simple corollary to this is that if $E(2,R)$ is the subgroup of $SL(2,R)$ generated by upper and lower triangular matrices, then $SL(2,R) = E(2,R)$. There is a corresponding notion for rings with involutions as well, originally explored by the second author\cite{Sheydvasser2021}: let $(R,\sigma)$ be a subring with involution of a division algebra; we say that it is $\sigma$-\emph{Euclidean} if there exists a function $\Phi: R\rightarrow W$ to some well-ordered set $W$ such that for all $a,b \in R$ with $b \neq 0$ and $a\sigma(b) \in R^+$, there exists some $q \in R^+$ such that $\Phi(a - bq) < \Phi(b)$. A $\sigma$-Euclidean ring has a corresponding $\sigma$-Euclidean algorithm, which retains many of the nice properties of the usual Euclidean algorithm; it is easily proved, for example, that if $(R,\sigma)$ is $\sigma$-Euclidean, then $SL^\sigma(2,R) = E(2,R)$, where $E(2,R)$ is the subgroup of $SL^\sigma(2,R)$ generated by upper and lower triangular matrices.

With these preliminaries, we can now fix some notation that we shall use throughout the rest of this paper. By $R(\ZZ)$, we shall mean either the ring of integers of an imaginary quadratic field (if $dim = 3$), a maximal $\ddagger$-order of a rational, definite quaternion algebra with orthogonal involution $\ddagger$ (if $dim = 4$), or a maximal order of a rational, definite quaternion algebra (if $dim = 5$). Furthermore, if $S$ is a commutative ring, then we shall write
\begin{align*}
        R(S) &= \begin{cases} R(\ZZ) \otimes_\ZZ S & \text{if $\text{dim} = 3,5$} \\ (R(\ZZ) \otimes_\ZZ S,\ddagger) & \text{if $\text{dim} = 4$} \end{cases} \\
        G(S) &= \begin{cases} SL(2,R(S)) & \text{if $\text{dim} = 3,5$} \\ SL^\ddagger(2,R(S)) & \text{if $\text{dim} = 4$,} \end{cases}
    \end{align*}
    
\noindent and $E(S)$ the subgroup of $G(S)$ generated by upper and lower triangular matrices. By some abuse of notation, we will write $R(S)^+$ to mean either $R(S)^+$ (if $dim = 4$) or just $R(S)$ (if $dim = 3,5$).

There is an interpretation of the above in terms of hyperbolic geometry, as follows. It is well-known that the isometry group of hyperbolic $dim$-space, $\Isom(\HH^{dim})$, is naturally isomorphic to $\Mob(dim - 1)$, the M\"obius group on $\RR^{dim - 1} \cup \{\infty\}$---that is, the group generated by reflections through $(dim - 2)$-spheres. This can be seen by taking the upper-half space model for $\HH^{dim}$ and identifying $\RR^{dim - 1} \cup \{\infty\}$ with the boundary $\partial \HH^{dim}$. Then the transformations on the boundary uniquely lift to isometries of hyperbolic space. If necessary, we can restrict to the orientation-preserving pieces of both groups, yielding the isomorphism $\Isom^0(\HH^{dim}) \cong \Mob^0(dim - 1)$. However, if $dim$ is small, then there are additional accidental isomorphisms. Specifically, if we take $H_\RR$ to be the Hamilton quaternions, then
    \begin{itemize}
        \item $\Mob^0(2) \cong PSL(2,\CC)$ via identifying $\RR^2$ with $\CC$,
        \item $\Mob^0(3) \cong PSL^\ddagger(2,H_\RR)$ via identifying $\RR^3$ with $H_\RR^+$,
        \item $\Mob^0(4) \cong PSL(2,H_\RR)$ via identifying $\RR^4$ with $H_\RR$.
    \end{itemize}
    
\noindent In each case, to define how the right hand group acts on $\RR^n \cup \{\infty\}$, we simply use the action
    \begin{align*}
        \begin{pmatrix} a & b \\ c & d \end{pmatrix}.\rho = (a\rho + b)(c\rho + d)^{-1}.
    \end{align*}
    
\noindent Note that $R(\ZZ)$ is a discrete subring of $\CC$ (if $dim = 3$), or a discrete subring with involution of $H_\RR$ (if $dim = 4$), or a discrete subring of $H_\RR$ (if $dim = 5$), and consequently $G(\ZZ)$ is a discrete subgroup of
    \begin{align*}
        G(\RR) \cong \begin{cases} SL(2,\CC) & \text{if $dim = 3$} \\ SL^\ddagger(2,H_\RR) & \text{if $dim = 4$} \\ SL(2,H_\RR) & \text{if $dim = 5$}, \end{cases}
    \end{align*}
        
\noindent which is to say that $G(\ZZ)/\{\pm 1\}$ is a Kleinian group. However, more than that, it is easy to see that $G(\ZZ)$ is an arithmetic group, with the immediate corollary that it a lattice. In this setting, this means the following: if one constructs a closed, connected fundamental domain $\mathcal{F}$ for the action of $G(\ZZ)/\{\pm 1\}$ on $\HH^{dim}$, then the hyperbolic volume of $\mathcal{F}$ must be finite. Note that since the volume of the fundamental domain of $E(\ZZ)$ must be $[G(\ZZ):E(\ZZ)]$ times the volume of the fundamental domain of $G(\ZZ)$, this means that $E(\ZZ)$ is a finite-index subgroup of $G(\ZZ)$ if and only if it is a lattice.

\section{Orders with Covering Vectors:}\label{section: covering vectors}

We start by classifying all of the orders that have covering vectors and specifying what those vectors are.

\begin{theorem}\label{thm: covering vector classification}
If $dim = 3$, then every ring of integers $R(\ZZ)$ has a covering vector $u$. Specifically, if $R(\QQ) = \QQ(\sqrt{-n})$ for a square-free integer $n > 0$, then $u = \sqrt{-n}$.

If $dim = 4$, then $R(\ZZ)$ has a covering vector $u$ if and only if it is isomorphic (as a ring with involution) to one of the orders enumerated in Table \ref{tab:dim4_covering_vectors}, and one can take $u = j$.

If $dim = 5$, then $R(\ZZ)$ has a covering vector $u$ if and only if it is isomorphic to one of the orders enumerated in Table \ref{tab:dim5_covering_vectors}, and one can take $u = ij/\gcd(i^2,j^2)$.
\end{theorem}

\begin{remark}
For small discriminants, there are some accidental isomorphisms between the orders in Table \ref{tab:dim4_covering_vectors}---for example, if we take
    \begin{align*}
        \OO_1 &= \ZZ \oplus \ZZ i \oplus \ZZ j \oplus \ZZ \frac{1 + i + j + ij}{2} \subset \left(\frac{-1,-1}{\QQ}\right) \\
        \OO_2 &= \ZZ \oplus \ZZ i \oplus \ZZ \frac{i + j}{2} \oplus \ZZ \frac{2 + 2j + ij}{4} \subset \left(\frac{-2,-2}{\QQ}\right),
    \end{align*}
    
\noindent then $\OO_1 \cong \OO_2$. However, if the discriminants are sufficiently large, this no longer occurs.
\end{remark}

	\begin{table}
	\begin{align*}
	\begin{array}{l|ll|l}
	R(\QQ) & R(\ZZ) & \text{Conditions} & |\discrd(R(\ZZ))| \\ \hline
	\left(\frac{-1,n}{\QQ}\right) &	\ZZ \oplus \ZZ i \oplus \ZZ j \oplus \frac{1 + i + j + ij}{2} & \text{if } n \equiv -1 \mod 4 & 2|n| \\
			& \ZZ \oplus \ZZ i \oplus \ZZ \frac{1 + i + j}{2} \oplus \ZZ \frac{j + ij}{2} & \text{if } n\equiv 2 \mod 4 & |n| \\
			& \begin{cases} \ZZ \oplus \ZZ i \oplus \ZZ \frac{1 + j}{2} \oplus \ZZ \frac{i + ij}{2} \\ \ZZ \oplus \ZZ i \oplus \ZZ \frac{i + j}{2} \oplus \ZZ \frac{1 + ij}{2} \end{cases} & \text{if } n \equiv 1 \mod 4 & |n|
		\\ \hline	
	\left(\frac{-2,n}{\QQ}\right) & \ZZ \oplus \ZZ i \oplus \ZZ \frac{1 + i + j}{2} \oplus \ZZ \frac{i + ij}{2} & \text{if } n \equiv -1 \mod 4 & 2|n| \\
			& \ZZ \oplus \ZZ i \oplus \ZZ \frac{1 + j}{2} \oplus \ZZ \frac{i + ij}{2} & \text{if } n \equiv 1 \mod 4 & 2|n| \\
			& \ZZ \oplus \ZZ i \oplus \ZZ \frac{i + j}{2} \oplus \ZZ \frac{2 + 2j + ij}{4} & \text{if } n/2 \equiv -1 \mod 8 & |n| \\
			& \ZZ \oplus \ZZ i \oplus \ZZ \frac{i + j}{2} \oplus \ZZ \frac{2 + ij}{4} & \text{if } n/2 \equiv -3 \mod 8 & |n| \\
			& \ZZ \oplus \ZZ i \oplus \ZZ \frac{2 + i + j}{4} \oplus \ZZ \frac{i - j + ij}{4} & \text{if } n/2 \equiv 3 \mod 8 & |n| \\
			& \ZZ \oplus \ZZ i \oplus \ZZ \frac{i + j}{4} \oplus \ZZ \frac{2 + ij}{4} & \text{if } n/2 \equiv 1 \mod 8 & |n|
		\\ \hline
	\left(\frac{-3,n}{\QQ}\right) & \ZZ \oplus \ZZ \frac{1 + i}{2} \oplus \ZZ j \oplus \ZZ \frac{j + ij}{2} & \text{if } 3\nmid n & 3|n| \\
	    & \ZZ \oplus \ZZ \frac{1 + i}{2} \oplus \ZZ j \oplus \ZZ \frac{3j + ij}{6} &\text{if } n/3 = -1 \mod 3 & |n|\\
	    & \ZZ \oplus \ZZ \frac{1 + i}{2} \oplus \ZZ \frac{i + j}{3} \oplus \ZZ \frac{3j + ij}{6} &\text{if } n/3 = 1 \mod 3 & |n|\\\hline
	\left(\frac{-7,n}{\QQ}\right) & \ZZ \oplus \ZZ \frac{1 + i}{2} \oplus \ZZ j \oplus \ZZ \frac{j + ij}{2} & \text{if } 7\nmid n & 7|n| \\
			& \ZZ \oplus \ZZ \frac{1 + i}{2} \oplus \ZZ j \oplus \ZZ \frac{7j + ij}{14} & \text{if } 7|n & |n| \\ \hline
	\left(\frac{-11,n}{\QQ}\right) & \ZZ \oplus \ZZ \frac{1 + i}{2} \oplus \ZZ j \oplus \ZZ \frac{j + ij}{2} & \text{if } 11\nmid n & 11|n| \\
			& \ZZ \oplus \ZZ \frac{1 + i}{2} \oplus \ZZ j \oplus \ZZ \frac{11j + ij}{22} & \hspace{-5pt}\begin{array}{l} \text{if } 11|n \text{ and } \\ \left(\frac{n/11}{11}\right) = -1 \end{array} & |n| \\
			& \ZZ \oplus \ZZ \frac{1 + i}{2} \oplus \ZZ \frac{3i + j}{11} \oplus \ZZ \frac{11j + ij}{22} & \hspace{-5pt}\begin{array}{l} \text{if } 11|n \text{ and } \\ n/11 \equiv -2 \mod 11 \end{array} & |n| \\
			& \ZZ \oplus \ZZ \frac{1 + i}{2} \oplus \ZZ \frac{4i + j}{11} \oplus \ZZ \frac{11j + ij}{22} & \hspace{-5pt}\begin{array}{l} \text{if } 11|n \text{ and } \\ n/11 \equiv 5 \mod 11 \end{array} & |n| \\
			& \ZZ \oplus \ZZ \frac{1 + i}{2} \oplus \ZZ \frac{2i + j}{11} \oplus \ZZ \frac{11j + ij}{22} & \hspace{-5pt}\begin{array}{l} \text{if } 11|n \text{ and } \\ n/11 \equiv 4 \mod 11 \end{array} & |n| \\
			& \ZZ \oplus \ZZ \frac{1 + i}{2} \oplus \ZZ \frac{5i + j}{11} \oplus \ZZ \frac{11j + ij}{22} & \hspace{-5pt}\begin{array}{l} \text{if } 11|n \text{ and } \\ n/11 \equiv 3 \mod 11 \end{array} & |n| \\
			& \ZZ \oplus \ZZ \frac{1 + i}{2} \oplus \ZZ \frac{i + j}{11} \oplus \ZZ \frac{11j + ij}{22} & \hspace{-5pt}\begin{array}{l} \text{if } 11|n \text{ and } \\ n/11 \equiv 1 \mod 11 \end{array} & |n|
	\end{array}
	\end{align*}
	\caption{All isomorphism classes of $R(\ZZ)$ with covering vectors in $dim = 4$.}
	\label{tab:dim4_covering_vectors}
	\end{table}
	
	\begin{table}
	    \begin{align*}
	        \begin{array}{l|l|l|l}
	        R(\QQ) & R(\ZZ) & |\discrd(R(\ZZ))| & \nrm(u) \\ \hline
	        \left(\frac{-1,-1}{\QQ}\right) & \ZZ \oplus \ZZ i \oplus \ZZ j \oplus \ZZ \frac{1 + i + j + ij}{2} & 2 & 1,2,3,6,10 \\
	        \left(\frac{-1,-3}{\QQ}\right) & \ZZ \oplus \ZZ i \oplus \ZZ \frac{i + j}{2} \oplus \ZZ \frac{1 + ij}{2} & 3 & 3,6 \\
	        \left(\frac{-2,-10}{\QQ}\right) & \ZZ \oplus \ZZ i \oplus \ZZ \frac{2 + i + j}{4} \oplus \ZZ \frac{2 + 2i + ij}{4} & 5 & 5,10 \\
	        \left(\frac{-1,-7}{\QQ}\right) & \ZZ \oplus \ZZ i \oplus \ZZ \frac{i + j}{2} \oplus \ZZ \frac{1 + ij}{2} & 7 & 7 \\
	        \left(\frac{-2,-26}{\QQ}\right) & \ZZ \oplus \ZZ i \oplus \ZZ \frac{2 + i + j}{4} \oplus \ZZ \frac{2 + 2i + ij}{4} & 13 & 13
	        \end{array}
	    \end{align*}
	    \caption{All isomorphism classes of $R(\ZZ)$ with covering vectors in $dim = 5$.}
	    \label{tab:dim5_covering_vectors}
	\end{table}
	
\begin{proof}
The $dim = 3$ case is obvious: there is only one choice of $u$ up to scaling, and the subspace orthogonal to it is $\ZZ$, which is of course covered by unit disks. If $dim = 4$, then $R(\ZZ)^+ \cap S_u$ must be an order of an imaginary quadratic field which is covered by open unit balls; this case was already handled in a paper of the second author\cite{Sheydvasser2019}. This gives the relevant entries in Table \ref{tab:dim4_covering_vectors}. Finally, for the $dim = 5$, we make use of successive minima. Recall that given a lattice $\Lambda \subset \RR^n$ (over $\ZZ$), its successive $i$-th successive minimum $\lambda_i(\Lambda)$ is the infimum of all $r > 0$ such that $\Lambda$ contains $i$ linearly independent vectors of length no more than $r$. Define as well the packing radius $\mu(\Lambda)$ as the smallest $r > 0$ such that closed $r$-balls centered at points in $\Lambda$ cover all of $\RR^n$. It is easy to check that $\mu(\Lambda) \geq \lambda_n(\Lambda)/2$. If $u \in R(\ZZ)_0^+$ is non-zero, then $\Lambda = S_u \cap R(\ZZ)$ is a $3$-dimensional lattice contained in the $4$-dimensional lattice $R(\ZZ)$, hence $\lambda_3(\Lambda) \geq \lambda_3(R(\ZZ))$. Therefore, if $u$ is a covering vector, then it must be that $\mu(\Lambda) \leq 1$, ergo $\lambda_3(R(\ZZ)) \leq 2$. We can also get a bound on $\lambda_4(R(\ZZ))$ by observing that if there are $3$ linearly-independent elements in the quaternion algebra $R(\QQ)$ with norm no more than $r$, then one of their products will be linearly-independent to them all, and will have norm no more than $r^2$. Consequently, $\lambda_4(R(\ZZ)) \leq \lambda_3(R(\ZZ))^2$. By Minkowski's second theorem,
    \begin{align*}
        \lambda_1(R(\ZZ))\lambda_2(R(\ZZ))\lambda_3(R(\ZZ))\lambda_4(R(\ZZ))\text{vol}(B_1) \geq \frac{2^4}{4!}\text{vol}\left(\RR^4/R(\ZZ)\right)
    \end{align*}
    
\noindent where $\text{vol}(B_1) = \pi^2/2$ is the volume of the unit ball in $\RR^4$. Noting that $\lambda_1(R(\ZZ)) = 1$, $1 \leq \lambda_2(R(\ZZ)) \leq \lambda_3(R(\ZZ))$, and $\text{vol}\left(\RR^4/R(\ZZ)\right) = |\disc(R(\QQ))|$, we obtain the bound
    \begin{align*}
        |\disc(R(\QQ))| \leq \frac{4!\lambda_3(R(\ZZ))^4(\pi^2/2)}{2^4} \leq 12\pi^2 < 119.
    \end{align*}
    
\noindent A quaternion algebra is uniquely determined by its discriminant, so that means that there are only finitely many possible choices of $R(\QQ)$ such that $R(\ZZ)$ has a covering vector. Moreover, if $R(\QQ)$ is rational, definite then it only has finitely many isomorphism classes of maximal orders, and these can be effectively enumerated---thus, there are only finitely many possible choices of $R(\ZZ)$ such that $R(\ZZ)$ has a covering vector. Finally, any potential covering vector is determined by the sub-lattice $S_u \cap R(\ZZ)$, which must have a basis of elements of norm no more than $4$---there are only finitely many such elements in $R(\ZZ)$, hence only finitely many such sub-lattices. These can be sifted through by hand if necessary, but it is significantly easier to write some code---in Magma, for example---which will enumerate this list and return only the orders which contain a covering vector. This yields the entries in Table \ref{tab:dim5_covering_vectors}, all of which are readily checked to have the claimed covering vectors.
\end{proof}

To what extent are these covering vectors unique? If we simply scale an existing covering vector, then we get another covering vector; this is uninteresting, as at most this changes the orientation of the initial oriented sphere $S_u$, which will just amount to a reflection of $\mathcal{A}_{R(\ZZ),u}$. With this in mind, we make the following definition.

\begin{definition}
Let $\OO,\OO'$ be maximal orders (if $dim = 3,5$) or maximal $\ddagger$-orders (if $dim = 4$). Let $u,u'$ be covering vectors for $\OO,\OO'$ respectively. We say that $(\OO,u)$ is \emph{equivalent} to $(\OO',u')$ if there exists a ring isomorphism $\varphi: \OO \rightarrow \OO'$ and a constant $c \in \RR^\times$ such that $\varphi(u) = c \varphi(u')$.
\end{definition}

\begin{theorem}\label{thm: covering vectors are excellent}
If $R(\ZZ)$ has a covering vector $u$, then $(R(\ZZ),u)$ is equivalent to one of orders enumerated in Theorem \ref{thm: covering vector classification} together with a covering vector $u$ such that $u \in R(\ZZ)_0^+$ has square-free norm, and such that the two-sided ideal $\mathfrak{I} \subset R(\ZZ)$ that it generates has $p|\mathfrak{I} \cap \ZZ| = p\nrm(u) = |\discrd(R(\ZZ))|$ where
    \begin{align*}
        p \in \begin{cases} \{1,2\} & \text{if } dim = 3 \\ \{1,2,3,7,11\} & \text{if } dim = 4 \\ \left\{\frac{1}{2},1,2,\frac{3}{2},3,5\right\} & \text{if } dim = 5. \end{cases}
    \end{align*}
\end{theorem}

\begin{proof}
If $dim = 3$, this is obvious: $R(\RR)_0$ is $1$-dimensional, hence any two elements of it are related by scaling. Otherwise, we note first that $u$ must be in $R(\QQ)$ up to scaling---if not, then $S_u \cap R(\ZZ)^+$ could not be a lattice of the necessary rank. Since scaling doesn't change equivalence, we may freely assume that $u \in R(\ZZ)_0^+$ and there does not exist any integer $k > 1$ such that $u/k \in R(\ZZ)_0^+$. If $dim = 4$, for $u$ to be a covering vector, $S_u \cap R(\ZZ)^+$ must be a Euclidean ring of integers of an imaginary quadratic field---this means that $R(\ZZ)$ must be isomorphic to one of the orders in Table \ref{tab:dim4_covering_vectors} and $u = \pm j$. But $\nrm(j) = n$, and so the relation between $|\discrd(R(\ZZ))|$ and $|\mathfrak{I} \cap \ZZ| = \nrm(j)$ can be read off from Table \ref{tab:dim4_covering_vectors}.

If $dim = 5$, we expand a little on our original proof that if $R(\ZZ)$ has a covering vector, then it must be isomorphic to an order in Table \ref{tab:dim5_covering_vectors}---our original proof used the fact that one can algorithmically enumerate all sub-lattices of rank $3$ with a basis of vectors with norm no more than $4$, and that any potential covering vectors must correspond to such sub-lattices. However, up to scaling, any such sub-lattice determines a unique covering vector. Therefore, when we enumerate these sub-lattices, we actually enumerate all of the covering vectors. There are too many of them to list them all compactly, but the norms of all these covering vectors $u$, normalized so that they are in $R(\ZZ)^+$ but $u/k \notin R(\ZZ)^+$ for any integer $k > 1$, are listed in Table \ref{tab:dim5_covering_vectors}. Together with the observation that $|\mathfrak{I} \cap \ZZ| = \nrm(u)$, this is enough to give us the theorem.
\end{proof}

With this theorem in mind, we make an additional definition.

\begin{definition}
A covering vector $u$ of $R(\ZZ)$ is \emph{normalized} if $u \in R(\ZZ)^+$ and for any integer $k > 1$, $u/k \notin R(\ZZ)$.
\end{definition}

Henceforth, we shall freely assume that $u$ is a normalized vector; as we now know, we do not lose any generality in so doing. Note that if we restrict $u$ to a normalized covering vector, then by Theorem \ref{thm: covering vectors are excellent}, $\nrm(u) = \Theta(|\discrd(R(\ZZ))|)$.

\section{Super-Apollonian Type Packings:}\label{section: super-apollonian}

We begin our construction of the Apollonian-type packings by first considering a ``super-packing"---a collection of oriented spheres that contains the desired sphere packing, but is in some ways easier to define. The term originated with the super Apollonian packing of Graham, Lagarias, Mallows, Wilkes, and Yan\cite{GLMWY2005}. To start, we define a convenient group.

\begin{definition}
For any $\alpha \in R(\ZZ)^+$, define
    \begin{align*}
        W(\alpha) = \begin{pmatrix} \alpha & 1 \\ -1 & 0 \end{pmatrix}.
    \end{align*}
    
\noindent Let $\mathcal{W}$ be the group generated by all elements $W(\alpha)$.
\end{definition}

\begin{remark}
Matrices of this form are sometimes called Cohn matrices\cite{Cohn1966}.
\end{remark}

\noindent Note that
    \begin{align*}
        W(\alpha) = \begin{pmatrix} 1 & -\alpha \\ 0 & 1 \end{pmatrix}\begin{pmatrix} 1 & 1 \\ 0 & 1 \end{pmatrix}\begin{pmatrix} 1 & 0 \\ -1 & 1 \end{pmatrix}\begin{pmatrix} 1 & 1 \\ 0 & 1 \end{pmatrix}
    \end{align*}
    
\noindent so certainly $W(a) \in E(\ZZ)$, hence $\mathcal{W}$ is a subgroup of $E(\ZZ)$. In fact, it is a finite-index subgroup.

\begin{theorem}\label{thm: Cohn decomposition}
For any $\gamma \in E(\ZZ)$, there exist $a_1, a_2, \ldots a_k \in R(\ZZ)^+$ such that $a_i \neq 0,\pm 1$ if $i \neq 1,k$ and
    \begin{align*}
        \gamma = W(a_1)W(a_2)\ldots W(a_k)\begin{pmatrix} \omega & 0 \\ 0 & \varpi \end{pmatrix}.
    \end{align*}
    
\noindent Consequently, $\mathcal{W}$ is a finite-index subgroup of $E(\ZZ)$.
\end{theorem}

\begin{remark}
Since $\mathcal{W}$ is a finite-index subgroup of $E(\ZZ)$, it follows that it is a lattice if and only if $E(\ZZ)$ is.
\end{remark}

\begin{proof}
Any upper triangular matrix in $G(\ZZ)$ can be decomposed as
    \begin{align*}
        \begin{pmatrix} * & * \\ 0 & * \end{pmatrix} = \begin{pmatrix} 1 & \alpha \\ 0 & 1 \end{pmatrix}\begin{pmatrix} \omega & 0 \\ 0 & \varpi \end{pmatrix}.
    \end{align*}

Similarly, any lower triangular matrix can be written as a product of a lower triangular matrix with $1$s on the diagonal, and a diagonal matrix. For any $\alpha \in R(\ZZ)^+$,
    \begin{align*}
        \begin{pmatrix} 1 & \alpha \\ 0 & 1 \end{pmatrix} &= -W(-\alpha)W(0) \\
        \begin{pmatrix} 1 & 0 \\ \alpha & 1 \end{pmatrix} &= -W(0)W(\alpha).
    \end{align*}
    
\noindent On the other hand,
    \begin{align*}
        \begin{pmatrix} \omega & 0 \\ 0 & \varpi \end{pmatrix}\begin{pmatrix} 1 & \alpha \\ 0 & 1 \end{pmatrix} &= \begin{pmatrix} 1 & \omega \alpha \varpi^{-1} \\ 0 & 1 \end{pmatrix}\begin{pmatrix} \omega & 0 \\ 0 & \varpi \end{pmatrix} \\
        \begin{pmatrix} \omega & 0 \\ 0 & \varpi \end{pmatrix}\begin{pmatrix} 1 & 0 \\ \alpha & 1 \end{pmatrix} &= \begin{pmatrix} 1 & 0 \\ \omega \alpha \varpi^{-1} & 1 \end{pmatrix}\begin{pmatrix} \omega & 0 \\ 0 & \varpi \end{pmatrix}.
    \end{align*}
    
\noindent Therefore, if we write an element in $E(\ZZ)$ as a product of upper and lower triangular matrices, we can move the non-trivial diagonal pieces to the right-hand side, and then rewrite everything in terms of $W(\alpha)$'s. This shows that the index of $\mathcal{W}$ in $R(\ZZ)$ can be at most $|R(\ZZ)^\times|^2$---since the unit group of $R(\ZZ)$ is necessarily finite, we conclude that $\mathcal{W}$ is a finite-index subgroup.
\end{proof}

\begin{definition}
For any non-zero $u \in R(\RR)_0$, define the \emph{restricted super-Apollonian packing} $\mathcal{S}_{R(\ZZ),u}$ to be the orbit of $S_u$ under the action of $\mathcal{W}$, where $S_u$ is the hyper-plane through $0$ in $\RR^{dim - 1} \cup \{\infty\}$ with normal vector $u$, oriented such that $u$ is in the interior of $S_u$. The \emph{super-Apollonian packing} is $\hat{\mathcal{S}}_{R(\ZZ),u} = \mathcal{S}_{R(\ZZ),u} \cup \mathcal{S}_{R(\ZZ),-u}$.
\end{definition}

\begin{remark}
Intuitively, the full super-Apollonian packing includes all of the spheres of the restricted one, but with both possible orientations.
\end{remark}

We will want to understand how spheres in $\hat{\mathcal{S}}_{R(\ZZ),u}$ can intersect. This is easiest to do in terms of inversive coordinates.

\begin{definition}
Given an oriented sphere $S$, we define its \emph{bend} $\kappa(S)$ to be 1/radius, taken to be
positive if $S$ is positively oriented, and negative otherwise. If $S$ is a plane, $\kappa(S) = 0$. The \emph{co-bend} $\kappa'(S)$ is the bend of the image of $S$ under the map $z \mapsto -z^{-1}$. If $S$ is not a plane, the \emph{bend-center} $\xi(S)$ is the product of the bend $\kappa(S)$ and the center of $S$. If $S$ is a plane, $\xi(S)$ is the unique unit normal vector to $S$, pointing in the direction of the interior of $S$.
\end{definition}

\noindent It is a standard exercise to check that
\begin{enumerate}
    \item $-\kappa(S)\kappa'(S) + \nrm(\xi(S)) = 1$ and
    \item $\kappa,\kappa',\xi$ are continuous functions.
\end{enumerate}

\noindent These are known as the \emph{inversive coordinates} for the oriented sphere---we shall henceforth write $\text{inv}(S) = (\kappa, \kappa',\xi)$. Inversive coordinates identify oriented spheres with the cone $q(\kappa,\kappa',\xi) = 1$ where $q(\kappa,\kappa',\xi) = -\kappa\kappa' + \nrm(\xi)$. This quadratic form has a corresponding bilinear form
    \begin{align*}
        b((s_1,t_1,\rho_1),(s_2,t_2,\rho_2)) &= \frac{1}{2}\left(q((s_1,t_1,\rho_1) + (s_2,t_2,\rho_2)) - q((s_1,t_1,\rho_1)) - q((s_2,t_2,\rho_2))\right) \\
        &= \frac{1}{2}\left(-s_1 t_2 - s_2 t_1 + \tr(\rho_1 \overline{\rho_2})\right)
    \end{align*}
    
\noindent which also has geometric meaning---if $S_1, S_2$ intersect, then $b(\text{inv}(S_1),\text{inv}(S_2)) = \cos(\theta)$, where $\theta$ is the angle of intersection between them. We say that $S_1, S_2$ \emph{intersect internally} if $\cos(\theta) = 1$; we say that they \emph{intersect externally} if $\cos(\theta) = -1$.

Given a matrix $\gamma \in G(\RR)$ and the inversive coordinates of an oriented sphere $S$, there is an easy way to compute the inversive coordinates of $\gamma.S$.

\begin{lemma}\label{actions coincide}
For any
    \begin{align*}
        \gamma = \begin{pmatrix} a & b \\ c & d \end{pmatrix} \in G(\RR),
    \end{align*}
    
\noindent and oriented sphere $S$ with inversive coordinates $(k_1,k_2,\beta)$, if
    \begin{align*}
        \begin{pmatrix} k_2' & \beta' \\ \overline{\beta'} & k_1' \end{pmatrix} = \begin{pmatrix} a & b \\ c & d \end{pmatrix} \begin{pmatrix} k_2 & \beta \\ \overline{\beta} & k_1 \end{pmatrix} \begin{pmatrix} \overline{a} & \overline{c} \\ \overline{b} & \overline{d} \end{pmatrix}
    \end{align*}
    
\noindent then $(k_1',k_2',\beta')$ are the inversive coordinates of $\gamma.S$.
\end{lemma}

\begin{proof}
For $\text{dim} = 3,4$, this was proved by Stange and the second author, respectively \cite{Sheydvasser2019,Stange2018}, so it remains to prove it for the case $\text{dim} = 5$. Note that
    \begin{align*}
        \sigma\left(\begin{pmatrix} a & b \\ c & d \end{pmatrix}\right) = \begin{pmatrix} \overline{a} & \overline{c} \\ \overline{b} & \overline{d} \end{pmatrix}
    \end{align*}
    
\noindent is an involution on $\Mat(2,H_\RR)$, the ring of $2 \times 2$ matrices with coefficients in $H_\RR$, and therefore
    \begin{align*}
        G(\RR) \times \Mat(2,H_\RR) &\rightarrow \Mat(2,H_\RR) \\
        (\gamma, M) &\mapsto \gamma M \sigma(\gamma)
    \end{align*}
    
\noindent is a well-defined action of $G(\RR)$ on $\Mat(2,H_\RR)$. The set of matrices in $\Mat(2,H_\RR)$ that are fixed by $\sigma$ are exactly the matrices of the form
    \begin{align*}
        \begin{pmatrix} k_2 & \beta \\ \overline{\beta} & k_1 \end{pmatrix}
    \end{align*}
    
\noindent where $k_1,k_2 \in \RR$---naturally, the action of $G(\RR)$ fixes this set. Additionally, if $-k_1 k_2 + \nrm(\beta) = 1$, then the same must be true of the image---this is because the Dieudonn\'e determinant is multiplicative and the Dieudonn\'e determinant of any matrix in the fixed set of $\sigma$ must be real. Therefore, we actually have an action on set of matrices in $\Mat(2,H_\RR)$ of the form
    \begin{align*}
        \begin{pmatrix} k_2 & \beta \\ \overline{\beta} & k_1 \end{pmatrix}
    \end{align*}
    
\noindent where $(k_1, k_2,\beta)$ are the inversive coordinates of an oriented sphere. It remains to show that this action matches the usual action of $G(\RR)$. Note that we only need to prove this for the case where $k_1 > 0$ since we can extend by linearity and continuity, and we only need to prove it for the generators of $G(\RR)$, which we can take to be
    \begin{align*}
        \begin{pmatrix} 1 & \tau \\ 0 & 1 \end{pmatrix}, \begin{pmatrix} \omega & 0 \\ 0 & \varpi \end{pmatrix}, \begin{pmatrix} 0 & 1 \\ -1 & 0 \end{pmatrix}
    \end{align*}
    
\noindent with $\nrm(\omega)\nrm(\varpi) = 1$. We calculate
    \begin{align*}
        \begin{pmatrix} 1 & \tau \\ 0 & 1 \end{pmatrix} \begin{pmatrix} k_2 & \beta \\ \overline{\beta} & k_1 \end{pmatrix} \begin{pmatrix} 1 & 0 \\ \overline{\tau} & 1 \end{pmatrix} &= \begin{pmatrix} * & \beta + k_1 \tau \\ \overline{\beta} + k_1\overline{\tau} & k_1 \end{pmatrix} \\
        \begin{pmatrix} \omega & 0 \\ 0 & \varpi \end{pmatrix} \begin{pmatrix} k_2 & \beta \\ \overline{\beta} & k_1 \end{pmatrix} \begin{pmatrix} \overline{\omega} & 0 \\ 0 & \overline{\varpi} \end{pmatrix} &= \begin{pmatrix} k_2 \nrm(\omega) & \omega \beta \overline{\varpi} \\ \varpi \overline{\beta}\overline{\omega} & k_1/\nrm(\omega) \end{pmatrix} \\
        \begin{pmatrix} 0 & 1 \\ -1 & 0 \end{pmatrix} \begin{pmatrix} k_2 & \beta \\ \overline{\beta} & k_1 \end{pmatrix} \begin{pmatrix} 0 & -1 \\ 1 & 0 \end{pmatrix} &= \begin{pmatrix} k_1 & -\overline{\beta} \\ -\beta & k_2 \end{pmatrix},
    \end{align*}
    
\noindent which indeed all match the expected action of $G(\RR)$.
\end{proof}

\begin{theorem}
Let $u \in R(\RR)^+$ be a unit vector. If
    \begin{align*}
        \gamma = \begin{pmatrix} a & b \\ c & d \end{pmatrix} \in G(\RR),
    \end{align*}
    
\noindent then
    \begin{align*}
        \kappa(\gamma.S_u) &= cu\overline{d} + d\overline{u}\,\overline{c} \\
        \kappa'(\gamma.S_u) &= au\overline{b} + b\overline{u}\,\overline{a} \\
        \xi(\gamma.S_u) &= au\overline{d} + b\overline{u}\,\overline{c}.
    \end{align*}
\end{theorem}

\begin{proof}
By direct computation,
    \begin{align*}
        \begin{pmatrix} a & b \\ c & d \end{pmatrix}\begin{pmatrix} 0 & u \\ \overline{u} & 0 \end{pmatrix}\begin{pmatrix} \overline{a} & \overline{c} \\ \overline{b} & \overline{d} \end{pmatrix} &= \begin{pmatrix} au\overline{b} + b\overline{u}\,\overline{a} & au\overline{d} + b\overline{u}\,\overline{c} \\ * & cu\overline{d} + d\overline{u}\,\overline{c} \end{pmatrix},
    \end{align*}
    
\noindent and the claim follows from Lemma \ref{actions coincide}.
\end{proof}

If $u$ is not a unit vector, this is of course no longer quite correct---there is a scaling factor that has to be introduced.

\begin{definition}
Let $u \in R(\ZZ)_0^+$ be a non-zero vector. For any $\gamma \in G(\ZZ)$, define
    \begin{align*}
        \kappa_u(\gamma) &= cu\overline{d} - du\overline{c} \\
        \kappa_u'(\gamma) &= au\overline{b} - bu \overline{a} \\
        \xi_u(\gamma) &= au\overline{d} - bu\overline{c}.
    \end{align*}
    
\noindent We shall call $\text{inv}_u(\gamma.S_u) = (\kappa_u(\gamma),\kappa_u'(\gamma),\xi_u(\gamma))$ the \emph{normalized inversive coordinates} of $\gamma.S_u$. 
\end{definition}

\begin{remark}
It is easily seen that $\text{inv}_u(\gamma.S_u) = \sqrt{\nrm(u)}\text{inv}(\gamma.S_u)$, and therefore
    \begin{align*}
        b\left(\text{inv}_u(\gamma.S_u),\text{inv}_u(\gamma'.S_u)\right) = \nrm(u)b\left(\text{inv}(\gamma.S_u),\text{inv}(\gamma'.S_u)\right).
    \end{align*}
    
\noindent Thus, the conversion between the two kinds of inversive coordinates is quite simple---however, the normalized ones will be more convenient.
\end{remark}

\begin{theorem}\label{thm: congruence restriction using normalized covering vectors}
Let $u$ be a normalized covering vector for $R(\ZZ)$. Then for all $\gamma \in \mathcal{W}$, $\text{inv}_u(\gamma.S_u) = (0,0,u) \mod \nrm(u)$.
\end{theorem}

\begin{proof}
Let $\mathfrak{I}$ be the two-sided ideal generated by $u$; recall that $|\mathfrak{I} \cap \ZZ| = \nrm(u)$. Write $\gamma = W(a_n)\ldots W(a_1)$ and define recursively $(\kappa_1,\kappa_1',\xi_1) = (0,0,u)$ and
    \begin{align*}
        \begin{pmatrix} \kappa_{i + 1}' & \xi_{i + 1} \\ \overline{\xi_{i + 1}} & \kappa_{i + 1} \end{pmatrix} = W(a_i) \begin{pmatrix} \kappa_i' & \xi_i \\ \overline{\xi_i} & \kappa_i \end{pmatrix} \overline{W(a_i)}^T.
    \end{align*}
    
\noindent We shall prove by induction that $(\kappa_i,\kappa_i',\xi_i') = (0,0,u) \mod \nrm(u)$. This is obviously true for the base case. For every subsequent step, note that
    \begin{align*}
        \begin{pmatrix} a & 1 \\ -1 & 0 \end{pmatrix}\begin{pmatrix} \kappa' & \xi \\ \overline{\xi} & \kappa \end{pmatrix}\begin{pmatrix} \overline{a} & -1 \\ 1 & 0 \end{pmatrix} &= \begin{pmatrix} \kappa'\nrm(a) + \tr(a\xi) + \kappa & -\kappa' a - \overline{\xi} \\ -\kappa'\overline{a} - \xi & \kappa' \end{pmatrix}.
    \end{align*}
    
\noindent By assumption, $\xi = u + \nrm(u)\alpha$ for some $\alpha \in R(\ZZ)$, hence $a\xi \in \mathfrak{I}$ and so $\tr(a\xi) \in \mathfrak{I} \cap \ZZ$. Putting this together, we get that
    \begin{align*}
        \begin{pmatrix} \kappa'\nrm(a) + \tr(a\xi) + \kappa & -\kappa' a - \overline{\xi} \\ -\kappa'\overline{a} - \xi & \kappa' \end{pmatrix} = \begin{pmatrix} 0 & u \\ -u & 0 \end{pmatrix} \mod \nrm(u),
    \end{align*}
    
\noindent which proves the theorem.
\end{proof}

\begin{corollary}
Let $u$ be a normalized covering vector for $R(\ZZ)$. If $\nrm(u) > 3$, then all intersections in $\mathcal{S}_{R(\ZZ),u}$ are internal.
\end{corollary}

\begin{remark}
As we previously established that $|\discrd(R(\ZZ))|$ grows linearly in $\nrm(u)$, so we know that this corollary implies that for all but finitely many orders $R(\ZZ)$, $\mathcal{S}_{R(\ZZ),u}$ has only internal intersections.
\end{remark}

\begin{remark}
This result is not quite sharp in the sense that even if $\nrm(u) \leq 3$, it is still possible that the intersections in $\mathcal{S}_{R(\ZZ),u}$ are internal---this happens for $\ZZ[\sqrt{-2}]$, for example.
\end{remark}

\begin{proof}
For any $\gamma.S_u, \gamma'.S_u \in \mathcal{S}_{R(\ZZ),u}$,
    \begin{align*}
        b\left(\text{inv}_u(\gamma.S_u),\text{inv}_u(\gamma'.S_u)\right) &= b\left(\text{inv}_u(S_u),\text{inv}_u(\gamma^{-1}\gamma'.S_u)\right) \\
        &= b\left((0,0,u),(0,0,u) + \nrm(u)(s,t,\alpha)\right)
    \end{align*}
    
\noindent for some $s,t \in \ZZ$ and $\alpha \in R(\ZZ)^+$, whence
    \begin{align*}
        b\left((0,0,u),(0,0,u) + \nrm(u)(s,t,\alpha)\right) &= \nrm(u) + \frac{\nrm(u)}{2}\tr(u\overline{\alpha}) \\
        &\in \nrm(u) + \frac{\nrm(u)^2}{2}\ZZ.
    \end{align*}
    
\noindent The statement immediately follows, since the norm of $u$ is square-free, so if $\nrm(u) > 3$, then in fact $\nrm(u) \geq 5$.
\end{proof}

It immediately follows from this that if the discriminant of $R(\ZZ)$ is large enough, then the tangency structure of $\mathcal{S}_{R(\ZZ),u}$ is extremely simple.

\begin{theorem}\label{thm: tangency structure}
Let $u$ be a normalized covering vector for $R(\ZZ)$ and $\nrm(u) > 3$. If $S_1,S_2 \in \mathcal{S}_{R(\ZZ),u}$ intersect, then they are internally tangent at some point $\rho \in R(\QQ)$. Furthermore, there exists $\gamma \in \mathcal{W}$ and $\tau \in R(\ZZ)^+$ such that $S_1 = \gamma.S_u$ and $S_2 = \gamma.(S_u + \tau)$.
\end{theorem}

\begin{proof}
Choose $S_1 = \gamma_1.S_u$, $S_2 = \gamma_2.S_u$. By the preceding result, we know that any intersections must be internal. Note that as well that the desired result is true for $S_1,S_2$ if and only if it is true for $S_u$, $\gamma_1^{-1}\gamma_2.S_u$, so we can in fact assume without loss of generality that $S_1 = S_u$, $S_2 = \gamma'.S_u$. Additionally, we may assume without loss of generality that $S_2$ is not a plane. Then the point of intersection $\rho$ will be the projection of $\xi(\gamma')/\kappa(\gamma')$ onto $S_u$, which is certainly in $R(\QQ)^+$. Since it is a rational point, we know that there is an element in $\text{Stab}_\mathcal{W}(S_u)$ that moves it to $\infty$. Why is this? By assumption, $S_u \cap R(\ZZ)^+$ is covered by open unit balls. Represent the intersection point $\rho$ as a pair $(\alpha,\beta) \in R(\ZZ)$ such that $\alpha\beta^{-1} = \rho$. We can then apply the Euclidean or $\ddagger$-Euclidean algorithm to this pair, which corresponds to acting by matrices
    \begin{align*}
        \begin{pmatrix} a & 1 \\ -1 & 0 \end{pmatrix}
    \end{align*}
    
\noindent with $a \in R(\ZZ)^+ \cap S_u$, and this will send $(\alpha,\beta)$ to $(g,0)$---this exactly corresponds to those matrices moving $\rho$ to $\infty$. Therefore, there exists $\gamma \in \mathcal{W}$ such that $\gamma.S_u = S_u$ and $\gamma.S_2$ is a plane parallel to $S_u$---the only possible such planes in $\mathcal{S}_{R(\ZZ),u}$ are $S_u + \tau$ for some $\tau \in R(\ZZ)^+$.
\end{proof}

\section{Apollonian-Type Packings:}\label{section: apollonian type}

We are finally ready to define the Apollonian-type packings.

\begin{definition}
Let $u$ be a normalized covering vector for $R(\ZZ)$ such that $\nrm(u) > 3$. Let $\tau$ be any element of $R(\ZZ)^+$ that is in the interior of $S_u$ and such that $S_u \cap R(\ZZ)^+$ and $\tau$ together generate the full lattice $R(\ZZ)^+$. We define $\Gamma_u$ to be the subgroup of $O^+(\text{dim},1)$ generated by translations $z \mapsto z + \omega$ for $\omega \in R(\ZZ)^+ \cap S_u$, and all reflections through spheres of radius $1$ centered at points in $R(\ZZ)^+ \cap S_u$ and $R(\ZZ)^+ \cap (S_u + \tau)$.

The corresponding \text{$(R(\ZZ),u)$-Apollonian packing} $\mathcal{A}_{R(\ZZ),u}$ is the orbit of $S_u$ and $-S_u + \tau$ under the action of $\Gamma_u$.
\end{definition}

\begin{remark}
The existence of at least one such element $\tau$ is guaranteed by the fact that $R(\ZZ)^+$ is a free abelian group and $S_u \cap R(\ZZ)^+$ is a subgroup of it; consequently, there must exist a basis of $R(\ZZ)^+$ that restricts to a basis of $S_u \cap R(\ZZ)^+$. The choice of $\tau$ must be unique up to addition by an element in $S_u \cap R(\ZZ)^+$. Consequently, this choice does not affect $\Gamma_u$.
\end{remark}

One useful observation is that if $\phi_z$ is a reflection through the unit sphere centered at $z \in R(\ZZ)^+$ and $T_w$ is a translation by $w \in R(\ZZ)^+$, then $\phi_z \circ T_w = T_w \circ \phi_{z - w}$. This means that any element in $\Gamma_u$ can be written in the form
    \begin{align*}
        \phi_{z_1} \circ \phi_{z_2} \circ \ldots \circ \phi_{z_n} \circ T_w
    \end{align*}
    
\noindent for some
    \begin{align*}
        z_1, \ldots z_n \in \left(R(\ZZ)^+ \cap S_u\right) \cup \left(R(\ZZ)^+ \cap (S_u + \tau)\right)
    \end{align*}
    
\noindent and $w \in R(\ZZ)^+ \cap S_u$. This allows us to prove some nice properties of the group $\Gamma_u$.

\begin{theorem}\label{thm: geometric finiteness}
Let $u$ be a normalized covering vector for $R(\ZZ)$ such that $\nrm(u) > 3$. Then $\Gamma_u$ admits a fundamental domain in $\HH^{dim}$ which is a convex, infinite volume, (generalized) hyperbolic polytope with finitely many sides---in particular, $\Gamma_u$ is geometrically finite.
\end{theorem}

\begin{proof}
First, consider the action of the translations $T_\alpha$ on $\RR^{\text{dim} - 1}$. This group of translations admits a Dirichlet domain $\mathcal{P}$, which will be a generalized Euclidean polytope with finitely many sides and infinite volume, since $R(\ZZ)^+ \cap S_u$ only has rank $dim - 2$. Then, if we take $\mathcal{P} \times (0,\infty) \subset \HH^{\text{dim}}$, this is a fundamental domain for the action of this translation group on $\HH^{\text{dim}}$.

Next, note that unit spheres centered at points in $R(\ZZ)^+$ intersect at dihedral angles which are of the form $\pi/k$ for some $k \in \ZZ$, since they must be $\sqrt{n}$ distance apart for some integer $n$. Consequently, the subgroup of $\Gamma_u$ generated by reflections $\phi_\alpha$ is a geometric reflection group. It has a simple fundamental domain $\mathcal{R} \subset \HH^{\text{dim}}$: take the set of points in $\HH^{\text{dim}}$ which are in the exterior of all of the unit spheres centered at points $\alpha$. We claim that 
    \begin{align*}
        \mathcal{F} = \mathcal{R} \cap \left(\mathcal{P} \times (0,\infty)\right)
    \end{align*}
    
\noindent is then a fundamental domain for $\Gamma_u$. Any point $\rho \in \HH^\text{dim}$ can be moved into the closure of $\mathcal{F}$ by first using the reflection $\phi_\alpha$ to move it into $\mathcal{R}$---but then, one can use translations $T_\alpha$ to move it into $\mathcal{F}$. On the other hand, we need to prove that if $\rho \in \mathcal{F}$, then there does not exist any non-identity element $\gamma \in \Gamma_u$ such that $\gamma(\rho) \in \mathcal{F}$. As we discussed above,
    \begin{align*}
        \gamma = \phi_{\alpha_1} \circ \ldots \circ \phi_{\alpha_n} \circ T_{\beta}
    \end{align*}
    
\noindent for some $\alpha_1,\ldots, \alpha_n \in (R(\ZZ)^+ \cap S_u) \cup (R(\ZZ)^+ \cap S_u + \tau)$ and  $\beta \in R(\ZZ)^+$. If $\gamma$ can be written without any reflections $\phi_\alpha$, then we use the fact that $\mathcal{P} \times (0,\infty)$ is a fundamental domain for the translation subgroup to conclude that $T_\beta(\rho) = \rho$ if and only if $T_\beta$ is the identity. Otherwise, we note that $T_\beta$ will move $\rho$ to some other point in $\mathcal{R}$, and then the reflections must necessarily move it out of $\mathcal{R}$, by virtue of the fact that $\mathcal{R}$ is a fundamental domain for the reflection group.

Examining the fundamental domain $\mathcal{F}$ that we have constructed, it is clear that it is a convex, generalized hyperbolic polytope with finitely many sides and infinite volume.
\end{proof}

\begin{theorem}\label{thm: zariski closure}
The Zariski closure of $\Gamma_u$ is $\Isom(\HH^{dim})$. Consequently, since $\Gamma_u$ has infinite co-volume, it is a thin group.
\end{theorem}

\begin{remark}
Here, we are thinking of $\Isom(\HH^{dim})$ as an algebraic group by exploiting the isomorphism $\Isom(\HH^{dim}) \cong O^+(dim,1)$. Thus, it makes perfect sense to talk about the Zariski closure of $\Gamma_u$ in this group.
\end{remark}

\begin{proof}
The Zariski closure of $\Gamma_u$ inside $\Isom(\HH^{dim})$ must be some Lie group---call it $\mathfrak{G}$. We shall show that its Lie algebra $\mathfrak{g}$ is the same as the Lie algebra as $\Isom(\HH^{dim})$. Since $\Isom^0(\HH^{dim})$ is simply-connected, this would imply that either $\mathfrak{G} = \Isom(\HH^{dim})$ or $\mathfrak{G} = \Isom^0(\HH^{dim})$---however, the latter can't happen since $\Gamma_u$ contains reflections.

The Lie algebra of $\Isom(\HH^{dim})$ is the Lie algebra of $G(\RR)$, which is easier to think about in this context: it is generated by all elements
    \begin{align*}
        \begin{pmatrix} 0 & \alpha \\ 0 & 1 \end{pmatrix}, \begin{pmatrix} 0 & 0 \\ \alpha & 1 \end{pmatrix}
    \end{align*}
    
\noindent with $\alpha \in R(\RR)^+$---this is a classical result for $dim = 3,5$ and was shown for $dim = 4$ by the second author \cite{Sheydvasser2020_preprint}. First, note that since there are infinitely many translations $z \mapsto z + \alpha$ with $\alpha \in R(\ZZ)^+ \cap S_u$, $\mathfrak{G}$ must actually contain all translations $z \mapsto z + \alpha$ with $\alpha \in R(\RR)^+ \cap S_u$, which are represented by matrices
    \begin{align*}
        \begin{pmatrix} 1 & \alpha \\ 0 & 1 \end{pmatrix}.
    \end{align*}
    
\noindent Thus
    \begin{align*}
        \begin{pmatrix} 0 & \alpha \\ 0 & 0 \end{pmatrix}, \begin{pmatrix} 0 & 0 \\ \alpha & 0 \end{pmatrix} \in \mathfrak{g}
    \end{align*}
    
\noindent for all $\alpha \in R(\RR)^+ \cap S_u$. Since $\mathfrak{g}$ must be closed under the Lie bracket, this also means that
    \begin{align*}
        \begin{pmatrix} \alpha & 0 \\ 0 & -\alpha \end{pmatrix} &= \left[\begin{pmatrix} 0 & \alpha \\ 0 & 0 \end{pmatrix}, \begin{pmatrix} 0 & 0 \\ 1 & 0 \end{pmatrix}\right] \\
        &= \begin{pmatrix} 0 & \alpha \\ 0 & 0 \end{pmatrix} \begin{pmatrix} 0 & 0 \\ 1 & 0 \end{pmatrix} - \begin{pmatrix} 0 & 0 \\ 1 & 0 \end{pmatrix}\begin{pmatrix} 0 & \alpha \\ 0 & 0 \end{pmatrix} \in \mathfrak{g}.
    \end{align*}

\noindent To progress further, note that $\mathfrak{G}$ acts on $\mathfrak{g}$ by conjugation, so if $\phi_\beta \in \mathfrak{G}$, then
    \begin{align*}
        \phi_\beta \circ \begin{pmatrix} 0 & \alpha \\ 0 & 0 \end{pmatrix} \circ \phi_\beta^{-1} \in \mathfrak{g}.
    \end{align*}
    
\noindent It is easy to check that
    \begin{align*}
        \phi_\beta(z) = \overline{(z - \beta)}^{-1} + \beta = \begin{pmatrix} -\beta & 1 - \nrm(\beta) \\ -1 & -\overline{\beta} \end{pmatrix}.(-\overline{z}),
    \end{align*}
    
\noindent and if we define $\psi(z) = -\overline{z}$, then one can use the generators of $G(\RR)$ to check that
    \begin{align*}
        \psi \circ \begin{pmatrix} a & b \\ c & d \end{pmatrix} \circ \psi^{-1} = \begin{pmatrix} \overline{d} & -\overline{b} \\ -\overline{c} & \overline{a} \end{pmatrix}.
    \end{align*}
    
\noindent Ergo,
    \begin{align*}
        \phi_\beta \circ \begin{pmatrix} 0 & \overline{\alpha} \\ 0 & 0 \end{pmatrix} \circ \phi_\beta^{-1} &= \begin{pmatrix} -\beta & 1 - \nrm(\beta) \\ -1 & -\overline{\beta} \end{pmatrix} \begin{pmatrix} 0 & -\alpha \\ 0 & 0 \end{pmatrix} \begin{pmatrix} -\beta & 1 - \nrm(\beta) \\ -1 & -\overline{\beta} \end{pmatrix}^{-1} \\
        &= \begin{pmatrix} \beta \alpha & -\beta \alpha \beta \\ \alpha & -\alpha\beta \end{pmatrix} \in \mathfrak{g}
    \end{align*}
    
\noindent for all $\alpha \in S_u \cap R(\RR)^+$ and $\beta \in (S_u \cap R(\ZZ)^+) \cup (S_u + \tau \cap R(\ZZ)^+)$. But then
    \begin{align*}
       \left[\begin{pmatrix} \beta \alpha & -\beta \alpha \beta \\ \alpha & -\alpha\beta \end{pmatrix}, \begin{pmatrix} 0 & 0 \\ 1 & 0 \end{pmatrix}\right] - \begin{pmatrix} \alpha & 0 \\ 0 & -\alpha \end{pmatrix} &= \begin{pmatrix} 0 & \beta \alpha + \alpha \beta \\ 0 & 0 \end{pmatrix} \in \mathfrak{g},
    \end{align*}
    
\noindent for all $\alpha \in S_u \cap R(\RR)^+$ and $\beta \in (S_u \cap R(\ZZ)^+) \cup (S_u + \tau \cap R(\ZZ)^+)$, so in particular
    \begin{align*}
        \begin{pmatrix} 0 & \tau \\ 0 & 0 \end{pmatrix} \in \mathfrak{g}.
    \end{align*}
    
\noindent Together with the other matrices we already showed are in $\mathfrak{g}$, this is sufficient to conclude that
    \begin{align*}
        \begin{pmatrix} 0 & \alpha \\ 0 & 0 \end{pmatrix} \in \mathfrak{g}
    \end{align*}
    
\noindent for all $\alpha \in R(\RR)^+$, and one can also obtain
    \begin{align*}
        \begin{pmatrix} 0 & 0 \\ \alpha & 0 \end{pmatrix} \in \mathfrak{g}
    \end{align*}
    
\noindent for all $\alpha \in R(\RR)^+$ similarly.
\end{proof}

\begin{theorem}\label{thm: limit set}
The limit set of $\Gamma_u$ is the closure of $\mathcal{A}_{R(\ZZ),u}$.
\end{theorem}

\begin{remark}
This result will be of critical importance later. It is also the most crucial place where we use the fact that $u$ is a covering vector---if it is not, and we attempt to define $\Gamma_u$ as before, then the result will in general be something significantly smaller than $\mathcal{A}_{R(\ZZ),u}$; rather than a sphere packing, one instead gets a disconnected set.
\end{remark}

\begin{proof}
Choose any point $\rho \in \HH^{dim}$. Since $\Gamma_u$ includes translations $z \mapsto z + \alpha$ for all $\alpha \in R(\ZZ)^+ \cap S_u$, $\infty$ must be an accumulation point of $\Gamma_u.\rho$. Thus, the limit set will be the closure of the orbit of $\Gamma_u.\infty$. However, this orbit contains all rational points in $S_u$---this is because for any such rational point, we can use either the Euclidean algorithm or the $\ddagger$-Euclidean algorithm to move it to $\infty$ using reflections $\phi_\alpha$ and translations $z \mapsto z + \alpha$ with $\alpha \in R(\ZZ)^+ \cap S_u$. Similarly, every rational point on $S_u + \tau$ is also in the limit set, because we can apply the same technique, but simply translating all our reflections over by $\tau$. Therefore, $S_u$, $-S_u + \tau$ and their orbits under $\Gamma_u$ are in the limit set. As the orbit of $\infty$ will be contained inside $\mathcal{A}_{R(\ZZ),u}$, we have thus showed that the limit set must be $\overline{\mathcal{A}_{R(\ZZ),u}}$.
\end{proof}

A consequence of this last result is that if we want to better understand $\Gamma_u$, it is worthwhile to understand $\mathcal{A}_{R(\ZZ),u}$. In fact, what we're going to show is that $\mathcal{A}_{R(\ZZ),u}$ is a very nice Apollonian-type packing which is contained inside the super-Apollonian packing we considered earlier.

\begin{lemma}
Let $u$ be a normalized covering vector for $R(\ZZ)$ with $\nrm(u) > 3$. Then $\mathcal{A}_{R(\ZZ),u} \subset \hat{\mathcal{S}}_{R(\ZZ),u}$.
\end{lemma}

\begin{proof}
Since $S_u$ and $-S_u + \tau$ are in $\hat{\mathcal{S}}_{R(\ZZ),u}$, we just have to show that the generators of $\Gamma_u$ preserve $\hat{\mathcal{S}}_{R(\ZZ),u}$. The translations are in $\mathcal{W}$. All other generators are of the form $z \mapsto \phi_0(z - \omega) + \omega$ for some $\omega \in R(\ZZ)^+$ and where $\phi_0(z) = \overline{z}^{-1}$. It is easy to check that $z \mapsto -\overline{z}$ preserves $\hat{\mathcal{S}}_{R(\ZZ),u}$ since it fixes $R(\ZZ)^+$ and since $z \mapsto -z^{-1}$ is in $\mathcal{W}$, it also preserves $\hat{\mathcal{S}}_{R(\ZZ),u}$. Since $\phi$ is their composition, it also preserves $\hat{\mathcal{S}}_{R(\ZZ),u}$ and thus so do all the generators.
\end{proof}

\begin{lemma}
Let $u$ be a normalized covering vector for $R(\ZZ)$ with $\nrm(u) > 3$. If $S_1,S_2 \in \mathcal{A}_{R(\ZZ),u}$ intersect, then they are tangent and there exists some $\gamma \in \Gamma_u$ which sends the point of tangency to $\infty$.
\end{lemma}

\begin{proof}
Since $\mathcal{A}_{R(\ZZ),u}$ is a subset of $\hat{\mathcal{S}}_{R(\ZZ),u}$, all intersections must be tangential and at rational points. Furthermore, we may assume without loss of generality that $S_1$ is either $S_u$ or $-S_u + \tau$. However, we already know that it is possible to use elements in $\Gamma_u$ that either stabilize $S_u$ or $-S_u + \tau$ to move any rational points on either of those two planes to $\infty$.
\end{proof}

\begin{lemma}
The only planes in $\mathcal{A}_{R(\ZZ),u}$ are $S_u$ and $-S_u + \tau$.
\end{lemma}

\begin{proof}
Suppose that there exists some $\gamma \in \Gamma_u$ such that $\gamma.S_u$ is a plane other than $S_u$. Because $\mathcal{A}_{R(\ZZ),u}$ is a subset of $\hat{\mathcal{S}}_{R(\ZZ),u}$, it must be $\pm S_u + k\tau$ for some $k \in \ZZ$. On the other hand, $p = \gamma^{-1}(\infty)$ must be a rational point on $\Gamma_u$, hence there is some $\gamma' \in \Gamma_u$ such that $\gamma'(S_u) = S_u$ and $\gamma'(\infty) = p$. This means that $\gamma''=\gamma \circ \gamma'$ has the property that $\gamma''(\infty) = \infty$ and $\gamma''(S_u) = \pm S_u + k\tau$. This presents a problem: it's not hard to see that the existence of such an element contradicts the fundamental domain for $\Gamma_u$ that we worked out. Therefore, there is no such element $\gamma \in \Gamma_u$. By a similar argument, we can show that there is no $\gamma \in \Gamma_u$ such that $\gamma(-S_u + \tau)$ is any plane other than $-S_u + \tau$.
\end{proof}

\begin{definition}
Let $S_1, S_2 \in \hat{\mathcal{S}}_{R(\ZZ),u}$. We say that they are \emph{immediately tangent} if they are externally tangent, their interiors do not intersect, and any other $S_3 \in \hat{\mathcal{S}}_{R(\ZZ),u}$ which is also tangent at the same point must be contained in either the interior of $S_1$ or the interior of $S_2$.
\end{definition}

\begin{remark}
This definition is equivalent to the one used by Stange in her definition of $K$-Apollonian packings\cite{Stange2015}.
\end{remark}

\begin{theorem}\label{thm: Apollonian structure}
Let $u$ be a normalized covering vector for $R(\ZZ)$ with $\nrm(u) > 3$. Then $\mathcal{A}_{R(\ZZ),u}$ is the minimal subset of $\hat{\mathcal{S}}_{R(\ZZ),u}$ that contains $S_u$ and is closed under immediate tangency. Furthermore, if $S_1,S_2 \in \mathcal{A}_{R(\ZZ),u}$ intersect, then they are immediately tangent.
\end{theorem}

\begin{proof}
First, note that $S_u$ and $-S_u + \tau$ are immediately tangent and moreover immediate tangency is preserved under the action of $\Gamma_u$. Second, if any two elements in $\mathcal{A}_{R(\ZZ),u}$ intersect, it is possible to move their point of tangency to $\infty$. However, since $S_u$ and $-S_u + \tau$ are the only possible planes, this proves that $S_1,S_2$ are immediately tangent.

Since we know that $\Gamma_u$ acts transitively on the rational points of $S_u$, it follows that $\mathcal{A}_{R(\ZZ),u}$ includes all spheres in $\hat{\mathcal{S}}_{R(\ZZ),u}$ that are immediately tangent to $S_u$; similarly, it includes all spheres that are immediately tangent to $-S_u + \tau$. Furthermore, it is easy to see that the action of the generators of $\Gamma_u$ either preserves $S_u$ (or $-S_u + \tau$) or sends it to a sphere immediately tangent to $-S_u + \tau$ (or $S_u$). Inductively, we can see that $\mathcal{A}_{R(\ZZ),u}$ will exactly consist of the spheres in $\hat{\mathcal{S}}_{R(\ZZ),u}$ that can be produced via immediate tangency.
\end{proof}

\begin{corollary}
All of the spheres in $\mathcal{A}_{R(\ZZ),u}$ other than $S_u$ and $-S_u + \tau$ have positive bend and are contained between $S_u$ and $-S_u + \tau$. The interiors of spheres do not intersect.
\end{corollary}

\begin{proof}
A sphere with negative bend cannot be immediately tangent to a sphere with positive bend. Since any sphere in $\mathcal{A}_{R(\ZZ),u}$ must be in a chain connected via immediate tangency, this means that all spheres that are not planes must have positive bend. If any such sphere were to be in the interior of $S_u$ or $-S_u + \tau$, then some sphere in this same chain would need to be internally tangent to one of those planes, which we know is impossible; therefore, all spheres must lie in their exteriors. But this means that there are no two spheres $S_1,S_2 \in \mathcal{A}_{R(\ZZ),u}$ with intersecting interiors, since we can use the action of $\Gamma_u$ to move $S_1$ to either $S_u$ or $-S_u + \tau$.
\end{proof}

\begin{corollary}\label{corollary: containment of superpacking}
Let $u$ be a normalized covering vector for $R(\ZZ)$ with $\nrm(u) > 3$. Then any sphere in $\hat{\mathcal{S}}_{R(\ZZ),u}$ is contained in sphere in $\mathcal{A}_{R(\ZZ),u}$.
\end{corollary}

\begin{proof}
We know that $\hat{\mathcal{S}}_{R(\ZZ),u}$ is connected via tangency because the action of the generators is to move $S_u$ to a sphere tangent to it. So, given a sphere $S$ in the super-packing, let $l(S)$ denote the length of the shortest sequence of spheres such that consecutive spheres are tangent, the first term is $S_u$, and the last term is $S$. So, we can induct on $S_u$ to prove that any sphere in $\hat{\mathcal{S}}_{R(\ZZ),u}$ is contained in sphere in $\mathcal{A}_{R(\ZZ),u}$. This is obviously true for $S_u$. For any other sphere $S$, look at the next sphere $S'$ in its chain. Since $l(S') = l(S) - 1$, $S'$ is contained inside some sphere in $\mathcal{A}_{R(\ZZ),u}$. Then either $S$ is contained inside this same sphere, or it is tangent to it. However, $\mathcal{A}_{R(\ZZ),u}$ is closed under immediate tangency, so $S$ must be contained in some sphere in $\mathcal{A}_{R(\ZZ),u}$ that is immediately tangent at that same point.
\end{proof}

This, finally, motivates our definition of the density of our Apollonian-type packings.

\begin{definition}
Let $u$ be a normalized covering vector for $R(\ZZ)$ such that $\nrm(u) > 3$. Then the \emph{density} of $\mathcal{A}_{R(\ZZ),u}$ is
    \begin{align*}
        \delta\left(\mathcal{A}_{R(\ZZ),u}\right) = \frac{1}{\text{vol}\left(R(\RR)^+/R(\ZZ)^+\right)}\sum_{\substack{S \in \mathcal{A}_{R(\ZZ),u}/\left(R(\ZZ)^+ \cap S_u\right) \\ S \neq S_u,-S_u + \tau}} \text{vol}(\text{Int}(S)),
    \end{align*}
    
\noindent where $\text{Int}(S)$ is the interior of $S$.
\end{definition}

\begin{remark}
Here, we are thinking of $R(\ZZ)^+ \cap S_u$ as an abelian group, which we can identify with all translations of $R(\RR)^+$ by elements in $R(\ZZ)^+ \cap S_u$. Thus, $\mathcal{A}_{R(\ZZ),u}/\left(R(\ZZ)^+ \cap S_u\right)$ should be understood as the set of orbits of $\mathcal{A}_{R(\ZZ),u}$ under this action.
\end{remark}

\begin{remark}
There is a natural sense in which $\mathcal{A}_{R(\ZZ),u}/\left(R(\ZZ)^+ \cap S_u\right)$ can be viewed as being a subset of $R(\RR)^+/R(\ZZ)^+$, which we are using here. Specifically, we know that all non-planes in the Apollonian-type packing are contained between $S_u$ and $-S_u + \tau$ and furthermore that this packing is preserved under the action of $S_u \cap R(\ZZ)^+$ by translations, since this group is contained in $\Gamma_u$.
\end{remark}

\begin{remark}
Because the interiors of spheres in $\mathcal{A}_{R(\ZZ),u}/\left(R(\ZZ)^+ \cap S_u\right)$ do not intersect, this really is a ``density" in the sense that this measures the proportion of the space that $\mathcal{A}_{R(\ZZ),u}$ occupies in $R(\RR)^+/R(\ZZ)^+$. In particular, it is necessarily between $0$ and $1$.
\end{remark}

The importance of the density is that if it is less than $1$, it immediately follows that $E(\ZZ)$ must be an infinite-index subgroup of $G(\ZZ)$.

\begin{lemma}
Let $u$ be a normalized covering vector for $R(\ZZ)$ such that $\nrm(u) > 3$. The density of $\mathcal{A}_{R(\ZZ),u}$ is less than $1$ if and only if there exists a ball $B$ such that $B$ does not intersect the interior of any sphere in $\mathcal{A}_{R(\ZZ),u}$.
\end{lemma}

\begin{proof}
Obviously, if there exists such a ball, then the density is less than $1$. On the other hand, suppose that there is no such ball, so the union of the interiors of spheres in $\mathcal{A}_{R(\ZZ)}$ is actually dense in $R(\RR)^+$. We proved that $\mathcal{A}_{R(\ZZ),u}$ is the limit set of $\Gamma_u$, which we proved to be finite and geometrically finite. By Patterson-Sullivan \cite{Patterson1976}\cite{Sullivan1979}, this set must have Hausdorff dimension strictly less than $dim - 1$ and so, in particular, must have Lebesgue measure zero. If the union of the interiors of spheres in $\mathcal{A}_{R(\ZZ)}$ is dense, then every point in $R(\RR)^+$ is either in the interior of some such sphere or it is in the limit set. Therefore, taking any polytope $P$ in $R(\RR)^+$, the volume of intersection of $P$ with the union of the interiors of spheres in $\mathcal{A}_{R(\ZZ)}$ is equal to the volume of $P$, which is to say that $\mathcal{A}_{R(\ZZ)}$ has density $1$.
\end{proof}

\begin{theorem}
Let $u$ be a normalized covering vector for $R(\ZZ)$ such that $\nrm(u) > 3$. Then if the density of $\mathcal{A}_{R(\ZZ),u}$ is less than $1$, $E(\ZZ)$ is an infinite index subgroup of $G(\ZZ)$.
\end{theorem}

\begin{proof}
Suppose that instead $E(\ZZ)$ is finite index; then, so is $\mathcal{W}$, which means that it is a lattice. Since it is a lattice, its limit set is all of $R(\RR)^+$. In particular, inside every open ball $B \subset R(\RR)^+$, there has to exist some $\gamma \in \mathcal{W}$ such that $\gamma.\infty \in B$. However, this is to say that there exists a sphere in $\hat{\mathcal{S}}_{R(\ZZ),u}$ that intersects $B$. However, by Corollary \ref{corollary: containment of superpacking}, any such sphere must be contained inside a sphere in $\mathcal{A}_{R(\ZZ),u}$. Therefore, any open ball $B \subset R(\RR)^+$ intersects the interior of a sphere in $\mathcal{A}_{R(\ZZ),u}$ which means, by the preceding lemma, that its density is $1$.
\end{proof}

\section{Behaviour of the Density of Apollonian-Type Packings:}\label{section: density estimates}
In order to obtain control on the density of our packing, we make use of two approaches: first using asymptotic estimates on integers represented by bends of spheres in a packing to get a general growth rate, and second using \textit{forbidden balls} to get explicit upper bounds. We first look in the direction of asymptotic estimates.

Recall the definition of the density of $\mathcal{A}_{R(\ZZ),u}$ given above,
    \begin{align*}
        \delta\left(\mathcal{A}_{R(\ZZ),u}\right) = \frac{1}{\text{vol}\left(R(\RR)^+/R(\ZZ)^+\right)}\sum_{\substack{S \in \mathcal{A}_{R(\ZZ),u}/\left(R(\ZZ)^+ \cap S_u\right) \\ S \neq S_u,-S_u + \tau}} \text{vol}(\text{Int}(S)),
    \end{align*}

\noindent First, let's work out the initial scaling factor.

\begin{theorem}\label{scalingfactor}
\begin{align*}
    \text{vol}\left(\RR^{n - 1}/R(\ZZ)^+\right) = \begin{cases} \frac{\sqrt{|\disc(R(\ZZ))|}}{2} & \text{if $\text{dim} = 3$} \\ \frac{\sqrt{|\disc(R(\ZZ))|}}{2\sqrt{|\iota(\disc(\ddagger))|}} & \text{if $\text{dim} = 4$} \\ \frac{\sqrt{|\disc(R(\ZZ))|}}{4} & \text{if $\text{dim} = 5$.} \end{cases}
\end{align*}
\end{theorem}

\begin{proof}
The volume of $\text{vol}\left(\RR^{n - 1}/R(\ZZ)^+\right)$ is equal to the volume of the fundamental parallelepiped spanned by the basis vectors $e_i$ of $R(\ZZ)^+$. The dot product in $\RR^n$ of $x,y \in R(\ZZ)$ can be computed as $\tr(x\overline{y})/2$---therefore, the square of the volume of this parallelepiped is
    \begin{align*}
        \frac{1}{2^{n - 1}}\left|\det\left(\tr(e_i \overline{e_j})\right)_{1\leq i,j \leq n - 1}\right| = \frac{|\disc(R(\ZZ)^+)|}{2^{n - 1}}.
    \end{align*}
    
\noindent For the cases $\text{dim} = 3$ and $\text{dim} = 5$, this resolves the issue entirely. It remains to compute $\disc(R(\ZZ)^+)$ if $\text{dim} = 4$, which we do by considering what happens at each completion $p$. If $p|\disc(R(\QQ))$ and $p \neq 2$, then
    \begin{align*}
        R(\ZZ_p)^+ = \ZZ_p \oplus \ZZ_p i \oplus \ZZ_p j
    \end{align*}
    
\noindent where $p|i^2$ and $p|j^2$ if and only if $p\nmid \iota(\disc(\ddagger))$. It is easy to compute that $\disc(R(\ZZ_p)^+) = p^2 \ZZ_p/\iota(\disc(\ddagger))$ in this case, which is to say that $\disc(R(\ZZ_p)^+) = \disc(R(\ZZ_p))/\iota(\disc(\ddagger))$. If $p = 2$, then there are exactly three possibilities as to the isomorphism class of $R(\ZZ_2)$ as a ring with involution, enumerated below.
    \begin{align*}
        \begin{array}{c|l|l}
        R(\QQ_2) & R(\ZZ_2)^+ & \disc(R(\ZZ_2)^+) \\ \hline
        \left(\frac{-1,-1}{\QQ_2}\right),\left(\frac{-1,3}{\QQ_2}\right) & \ZZ_2 \oplus \ZZ_2 i \oplus \ZZ_2 j & 2^3\ZZ_2 \\
        \left(\frac{-3,\pm 2}{\QQ_2}\right),\left(\frac{-3,\pm 6}{\QQ_2}\right) & \ZZ_2 \oplus \ZZ_2 \frac{1 + i}{2} \oplus \ZZ_2 j & 2^2 \ZZ_2 \\
        \left(\frac{2,6}{\QQ_2}\right) & \ZZ_2 \oplus \ZZ_2 i \oplus \ZZ_2 \frac{i + j}{2} & 2^3 \ZZ_2.
        \end{array}
    \end{align*}
    
\noindent In each case, we have $\disc(R(\ZZ_2)^+) = 2\disc(R(\ZZ_2))/\iota(\disc(\ddagger))$. Now, consider the case where $p\nmid \disc(R(\QQ))$. Unless $p \neq 2$ and $\disc(\ddagger) = -1 \mod 4$, up to isomorphism
    \begin{align*}
        R(\ZZ_p)^+ = \left\{\begin{pmatrix} a & b \\ b\lambda & d \end{pmatrix}\middle|a,b,d \in \ZZ_p\right\}.
    \end{align*}

\noindent for some $\lambda \in \ZZ_p$ so that $\disc(\ddagger) = \lambda \left(\ZZ_p^\times\right)^2$. The discriminant of this lattice is $2\lambda \ZZ_p = 2\iota(\disc(\ddagger)) = 2 \disc(R(\ZZ_p))/\iota(\disc(\ddagger))$. Finally, if $p = 2$ and $\disc(\ddagger) = -1 \mod 4$, then it is possible that instead
    \begin{align*}
        R(\ZZ_2) = \begin{pmatrix} 1 & -1 \\ 1 & 1 \end{pmatrix} \Mat(2,\ZZ_2) \begin{pmatrix} 1 & -1 \\ 1 & 1 \end{pmatrix}^{-1},
    \end{align*}
    
\noindent in which case
    \begin{align*}
        R(\ZZ_2)^+ = \left\{\begin{pmatrix} a + \frac{c - b \lambda}{\lambda - 1} & \frac{(b - c)\lambda}{\lambda - 1} \\ \frac{b - c}{\lambda - 1} & a + c + \frac{c - b}{\lambda - 1} \end{pmatrix}\middle|a,b,c \in \ZZ_2\right\}.
    \end{align*}

\noindent for some $\lambda \in \ZZ_2^\times$ so that $\disc(\ddagger) = \lambda \left(\ZZ_2^\times\right)^2$. Then $\disc(R(\ZZ_2)^+) = 8\lambda/(\lambda - 1)^2\ZZ_2 = 2\ZZ_2$. We see that in absolutely every case, $\disc(R(\ZZ_p)^+) = 2\disc(R(\ZZ_p))/\iota(\disc(\ddagger))$, and consequently
    \begin{align*}
        \frac{|\disc(R(\ZZ)^+)|}{2^{4 - 1}} = \frac{|\disc(R(\ZZ))|}{4|\iota(\disc(\ddagger))|}.
    \end{align*}
    
\noindent Taking a square root, we have the claimed result.
\end{proof}

\noindent To handle the sum, let $\hat{S}$ be the plane with normal vector $u \in R(\ZZ)_0^+$ such that $u$ is of minimal scaling. Then the bend of any sphere $S$ in the packing is in $2^l/\sqrt{\nrm(u)}\ZZ$ for some non-negative $l$. Consequently,
    \begin{align*}
        \sum_{S \in \mathcal{S}_{\Gamma_u,\hat{S}}^*/R(\ZZ)^+} \text{vol}\left(\text{Int}(S)\right) &= \sum_{S \in \mathcal{S}_{\Gamma_u,\hat{S}}^*/R(\ZZ)^+} \frac{\pi^{n/2}}{\Gamma(n/2 + 1)\kappa(S)^{n - 1}} \\
        &= \frac{\pi^{n/2}\nrm(u)^{(n - 1)/2}}{2^l\Gamma(n/2 + 1)} \sum_{k = 1}^\infty \frac{\mathcal{R}_{\Gamma_u,\hat{S}}(k)}{k^{n - 1}} 
    \end{align*}
    
\noindent where $\mathcal{R}_{\Gamma_u,\hat{S}}(k)$ is the number of spheres in $\mathcal{S}_{\Gamma_u,\hat{S}}/R(\ZZ)^+$ with bend $2^l k/\sqrt{\nrm(u)}$ and $\Gamma(\cdot)$ denotes the Gamma function. Counting the number of spheres in a packing with bend less than a parameter $T$ has been studied extensively over the last few decades (\cite{KontOh2011},\cite{LeeOh2013},\cite{OhShah2013},\cite{OhShah2016} and many more). In order to gain control on $\mathcal{R}_{\Gamma_u,\hat{S}}(k)$, we first remark on a few properties of $\Gamma_u$ that will allow us to apply results on the asymptotic behaviour of $\mathcal{R}_{\Gamma_u,\hat{S}}(k)$. (For more information on the Bowen-Margulis-Sullivan measure $m^\BMS$, skinning size $\sk_{\Gamma_u}(w_0)$, and critical exponent $\alpha$, see \cite{OhShah2012},\cite{OhShah2013},\cite{OhShah2016}.)
    \begin{enumerate}
        \item By Theorem \ref{thm: geometric finiteness},
        $\Gamma_u$ is geometrically finite, and it has been shown by Sullivan \cite{Sullivan1979} that geometric finiteness implies $|m^{\BMS}|<\infty$.
        \item The critical exponent $\alpha$ has been shown to be the Hausdorff dimension of the packing generated by $\Gamma_u$ which is strictly greater than 1. By Theorem 1.5 of Oh and Shah \cite{OhShah2013}, this implies $\sk_{\Gamma_u}(w_0)<\infty$.
        \item Since the orbit $w_0\Gamma_u$ is infinite, $\sk_{\Gamma_u}(w_0)>0$ and the constant $c(\Gamma_u)$ in Theorem \ref{OhShahthm} is positive.
    \end{enumerate}

We use the following result of Oh and Shah summarized below.
\begin{theorem}[Theorem 1.2 \cite{OhShah2013}]\label{OhShahthm}
Let $\Gamma < SO(n,1)$ be a non-elementary (in the Fuchsian-sense) discrete subgroup with $|m^{\BMS}| < \infty$, and $V$ a vector space on which $\Gamma$ acts from the right. Suppose that $w_0\Gamma$ is discrete and that its skinning size $\sk_\Gamma(w_0) := |\mu^{\PS}|$ is finite. Let $\alpha$ be the critical exponent of $\Gamma$. Then there exists $\lambda\in \NN$ such that for any norm $\Vert\cdot\Vert$ on $V$ invariant under the Cartan involution stabilizer of $SO(n,1)$, we have
    \begin{align*}
        \lim_{T \to \infty} \frac{\# \{w \in w_0 \Gamma| \Vert w \Vert <T \}}{T^{\alpha/\lambda}} = \frac{|\nu_0| \cdot \sk_{\Gamma}(w_0) }{\alpha \cdot |m^{\BMS}| \cdot \Vert w_0^\lambda \Vert^{\alpha / \lambda}} = \frac{c(\Gamma)}{\alpha}.
    \end{align*}
\end{theorem}
Considering only the machinery we need at hand, we obtain $\mathcal{R}_{\Gamma_u,\hat{S}}(k) \sim \frac{c(\Gamma_u)}{\alpha}T^{\alpha/\lambda} - \frac{c(\Gamma_u)}{\alpha}(T-1)^{\alpha/\lambda}\sim \frac{c(\Gamma_u)}{\alpha}T^{\alpha/\lambda-1}$ for $T=2^l k/\sqrt{\nrm(u)}$. In conjunction with the above calculations, we have
    \begin{align*}
        \sum_{S \in \mathcal{S}_{\Gamma_u,\hat{S}}^*/R^+(\ZZ)} \text{vol}\left(\text{Int}(S)\right)
        &= \frac{\pi^{n/2}\nrm(u)^{(n - 1)/2}}{2^l\Gamma(n/2 + 1)} \sum_{k = 1}^\infty \frac{\mathcal{R}_{\Gamma_u,\hat{S}}(k)}{k^{n - 1}} 
        \\
        &= \frac{c(\Gamma_u)}{\alpha} \frac{\pi^{n/2}\nrm(u)^{\frac{n-\alpha/\lambda}{2}}}{2^{l(2-\alpha/\lambda)}\Gamma(\frac{n}{2}+1)}\zeta\left(n-\frac{\alpha}{\lambda}\right) +O(1).
    \end{align*}

Excluding $c(\Gamma_u)$ (for which we only know $0<c(\Gamma_u)<\infty$), we are able to compute or effectively estimate all quantities present in the above equation. In the direction of computing the constant present in the representation number, the case of $n=3$ was handled independently by Lee-Oh and Vinogradov (with different degrees on their error terms) for Apollonian circle packings. In particular, for a circle packing $\mathcal{P}$ and $\mathcal{N}_\mathcal{P}(T)=\{C \in \mathcal{P}: Curv(C)<T\}$, it was shown by Vinogradov \cite{vinogradov_2013} that for any $\varepsilon>0$, one has
\begin{align*}
    \mathcal{N}_\mathcal{P}(T) = c_\mathcal{P}\cdot T^\alpha + O_\varepsilon(T^{\frac{128\delta+s_1}{129}+\varepsilon})
\end{align*} where the constant $c_\mathcal{P}$ is given explicitly in \cite[Remark 2.2.2]{vinogradov_2013} in terms of the Patterson-Sullivan measure on $\partial \HH^3$ and its quotients, and $\alpha \approx 1.30568$ is a universal constant of the circle packing computed by McMullen \cite{mcmullen_1998}, where the universality of the constant comes from all circle packings being related by M\"{o}bius transformations. The methods of Lee and Oh relate the constant $c_\mathcal{P}$ to $\mathcal{H}^\alpha(Res(\mathcal{P}))$, the $\alpha$-dimensional Hausdorff measure of the residual set of the packing $\mathcal{P}$, and a constant $c>0$ given in terms of the base eigenfunction on the Laplacian which is independent of the packing. It was remarked by Lee and Oh that their methods should extend to be able to handle packings in higher dimensions; however, as the dimension increases the computations needed to understand the spectral theory associated to the symmetry group of the packing get quite intricate. For the case of a general packing $\cP$ in the plane, Oh and Shah \cite{OhShah2012} gave a similar asymptotic formula (without an error term) in terms of the skinning size, the BMS measure, and the measure of the region that the circles are taken to intersect. For our uses, non-zero finiteness of the constant is strong enough to tackle the problem at hand.

By combining the previous calculation with Theorem \ref{scalingfactor}, we obtain
    \begin{align*}
        \delta\left(\mathcal{A}_{R(\ZZ),u}\right) &= \frac{1}{\text{vol}\left(\RR^{n - 1}/R(\ZZ)^+\right)}\sum_{S \in \mathcal{S}_{\Gamma_u,\hat{S}}^*/R(\ZZ)^+} \text{vol}\left(\text{Int}(S)\right) \\ &= \begin{cases}
        \frac{c(\Gamma_u)}{\alpha} \frac{\pi^{3/2}\nrm(u)^{\frac{3-\alpha/\lambda}{2}}}{2^{l(2-\alpha/\lambda)}\Gamma(\frac{5}{2})}\zeta\left(3-\frac{\alpha}{\lambda}\right)
        \frac{2}{\sqrt{|\disc(R(\ZZ))|}} & \text{if $\text{dim} = 3$} \\ \frac{c(\Gamma_u)}{\alpha} \frac{\pi^{2}\nrm(u)^{\frac{4-\alpha/\lambda}{2}}}{2^{l(2-\alpha/\lambda)}\Gamma(3)}\zeta\left(4-\frac{\alpha}{\lambda}\right) \frac{2\sqrt{|\iota(\disc(\ddagger))|}}{\sqrt{|\disc(R(\ZZ))|}} & \text{if $\text{dim} = 4$} \\ \frac{c(\Gamma_u)}{\alpha} \frac{\pi^{5/2}\nrm(u)^{\frac{5-\alpha/\lambda}{2}}}{2^{l(2-\alpha/\lambda)}\Gamma(\frac{7}{2})}\zeta\left(5-\frac{\alpha}{\lambda}\right) \frac{4}{\sqrt{|\disc(R(\ZZ))|}} & \text{if $\text{dim} = 5$} \end{cases}\\
        &\rightarrow 0 \qquad \mathrm{ as }\  |\disc(R(\ZZ))|\to \infty.
    \end{align*}
This asymptotic approach gives us an explicit growth rate in terms of the discriminant at which the density decreases but the lack of control on the constant makes giving precise estimates difficult. To this end, we look into cases where we can say when there is a sphere that intersects the fundamental region without intersecting $\mathcal{A}_{R(\ZZ),u}$. This will provide us with an upper bound on the density. By then looking at the orbit of this sphere, we can get closer and closer to the true value of $\delta\left(\mathcal{A}_{R(\ZZ),u}\right)$.

\section{Forbidden Balls:}\label{section: forbidden balls}

We begin this section with a simple observation.

\begin{theorem}\label{thm: forbidden balls force infinite index}
If there exists $u \in R(\ZZ)_0^+$ and an open ball $B \subset \RR^{n - 1}$ such that $\gamma.S_u$ is not contained in $B$ for any $\gamma \in \mathcal{W}$, then $E(\ZZ)$ is infinite index in $G(\ZZ)$.
\end{theorem}

\begin{proof}
Suppose that $E(\ZZ)$ is a finite index subgroup of $G(\ZZ)$. Then, equivalently, so is $\mathcal{W}$, and all three of these groups are lattices. As they are lattices, their limit sets are all of $\RR^{n - 1}$, and in particular for any open ball $B \subset \RR^{n - 1}$, there must exist $\gamma \in \mathcal{W}$ such that $\gamma.\infty \in B$. If we fix a normal vector $u \in R(\ZZ)_0^+$, this means that $\gamma.S_u$ intersects $B$. Of course, there are infinitely many spheres tangent to any such sphere at points in $B$, and those spheres can have arbitrarily large bends. Therefore, for any open ball $B$, there must exist some $\gamma' \in \mathcal{W}$ such that $\gamma'.S_u$ is completely contained in $B$.
\end{proof}

One consequence of this is that we want to prove that $E(\ZZ)$ is infinite index in $G(\ZZ)$, it suffices to find a sphere $S$ that does not intersect any sphere in $\mathcal{S}_{R(\ZZ),u}$---after all, if there exists some sphere in $\mathcal{S}_{R(\ZZ),u}$ contained in the interior of that sphere, then there must exist some sequence of spheres in $\mathcal{S}_{R(\ZZ),u}$ such that each intersects the subsequent one, such that this chain starts in the interior of the sphere and ends in its exterior. We call such an object a \emph{forbidden ball}.

\begin{definition}
For any $u \in R(\ZZ)_0^+$, a sphere $S$ in $\RR^{n - 1} \cup \{\infty\}$ is a \emph{forbidden ball} for $\mathcal{S}_{R(\ZZ),u}$ if its interior does not intersect any sphere in $\mathcal{S}_{R(\ZZ),u}$.
\end{definition}

The definition of a forbidden ball is a generalization of ``ghost circles" and ``ghost spheres" found in Katherine Stange's work \cite{Stange2018} and the work of the second author \cite{Sheydvasser2019}, in the sense that ghost circles/spheres are always forbidden balls, but not vice versa. All forbidden balls that we shall construct shall come around the same way, namely by looking for oriented spheres orthogonal to unit spheres centered at points in $R(\ZZ)^+$.

\begin{theorem}\label{thm:n=3forbiddenball}
For $\text{dim} = 3$, if $|\disc(R(\ZZ))| > 11$, then if we take $u = \sqrt{-|\disc(R(\ZZ))|}$, the circle with center
    \begin{align*}
        \begin{cases} \frac{1}{2} + \frac{\sqrt{-|\disc(R(\ZZ))|}}{4} & \text{if $|\disc(R(\ZZ))| = 0 \mod 4$} \\ \frac{1}{2} + \frac{|\disc(R(\ZZ))|-1}{4\sqrt{-|\disc(R(\ZZ))|}} & \text{if $|\disc(R(\ZZ))| = 1 \mod 4$} \end{cases}
    \end{align*}
    
\noindent and radius
    \begin{align*}
        \begin{cases} \frac{\sqrt{|\disc(R(\ZZ))| - 12}}{4} & \text{if $|\disc(R(\ZZ))| = 0 \mod 4$} \\ \frac{\sqrt{|\disc(R(\ZZ))|^2 - 14|\disc(R(\ZZ))| + 1}}{4\sqrt{|\disc(R(\ZZ))|}} & \text{if $|\disc(R(\ZZ))| = 1 \mod 4$} \end{cases}
    \end{align*}
    
\noindent is a forbidden ball for $\mathcal{S}_{R(\ZZ),u}$.
\end{theorem}

\begin{proof}
These are the ghost circles defined by Stange, and so we know that they don't intersect any circle $C \in G(\ZZ).S_u$\cite[Lemmas 7.7, 7.8]{Stange2018}. Consequently, any such ghost circle does not intersect any circle in $\mathcal{S}_{R(\ZZ),u}$, which is a subset.
\end{proof}

\begin{corollary}\label{cor:n=3forbiddensity}
For dim=3, the volume of the fundamental parallelogram with the forbidden ball removed is
    \begin{align*}
        \begin{cases}
            4 \frac{1}{\sqrt{|\disc(R(\ZZ)|}}+O\left(\frac{1}{\sqrt{|\disc(R(\ZZ)|}^3}\right) & \mathrm{if }\ |\disc(R(\ZZ))|=0 \mod 4 \\
            6 \frac{1}{\sqrt{|\disc(R(\ZZ)|}}+O\left(\frac{1}{\sqrt{|\disc(R(\ZZ)|}^3}\right) & \mathrm{if }\ |\disc(R(\ZZ))|=1 \mod 4 
        \end{cases}
    \end{align*}   
\end{corollary}
\begin{proof}
Let $d=|\disc(R(\ZZ))|$. In the case of $d=0 \mod 4$, $\mathfrak{o}_K$ has integral basis $\{1,\frac{\sqrt{-d}}{2}$\}, which gives us a fundamental rectangle of width 1 and height $\frac{\sqrt{d}}{2}$. By Theorem \ref{thm:n=3forbiddenball}, the forbidden ball in our packing has center $\frac{1}{2}+\frac{\sqrt{-d}}{4}$ and radius $\frac{\sqrt{d - 12}}{4}$. This ball intersects the boundary of the fundamental region when $x=0$, giving us the values of $y$ on the forbidden ball to be
    \begin{align*}
        \left(0-\frac{1}{2}\right)^2&+\left(y-\frac{\sqrt{d}}{4}\right)^2 = \frac{d-12}{16} \\
        &\Rightarrow y = \pm \frac{
        \sqrt{d-16}}{4}+\frac{\sqrt{d}}{4}
    \end{align*}
The regions contained between these two values of $y$ and $0 \leq x \leq 1$ lie strictly in the forbidden ball, but the region below and above these $y$ values are an upper bound of the area outside the forbidden ball. Explicitly, this is given by
    \begin{align*}
        &\left(-\frac{
        \sqrt{d-16}}{4}+\frac{\sqrt{d}}{4}-0\right) + \left(\frac{\sqrt{d}}{2} - \frac{
        \sqrt{d-16}}{4}-\frac{\sqrt{d}}{4}\right) \\
        &= 2\left(\frac{\sqrt{d}}{4}-\frac{
        \sqrt{d-16}}{4}\right) = \frac{1}{2}\left(8\frac{1}{\sqrt{d}}+32 \frac{1}{\sqrt{d}^{ 3}}+O\left(\frac{1}{\sqrt{d}^5}\right)   \right) \\
        &=\ 4\frac{1}{\sqrt{d}}+O\left(\frac{1}{\sqrt{d}^3}\right)
    \end{align*}
For the case of $d=1 \mod 4$, $\mathfrak{o}_K$ has integral basis $\{1,\frac{1+\sqrt{-d}}{2}$\}, which gives us a fundamental parallelogram generated by the vectors 1 and $\frac{1+\sqrt{-d}}{2}$. The forbidden ball with center $\frac{1}{2} + \frac{d-1}{4\sqrt{-d}}$ and radius $\frac{\sqrt{d^2 - 14d + 1}}{4\sqrt{d}}$ fully contains the region in the parallelogram whose $y$ coordinates are those where the ball intersects the boundary of the fundamental parallelogram on the right side of the parallelogram. Points on this edge satisfy $y=\sqrt{d}(x-1)$, giving us
    \begin{align*}
        &\left(\frac{y}{\sqrt{d}}+\frac{1}{2}   \right)^2 + \left(y- \frac{d-1}{4\sqrt{d}}    \right)^2 = \frac{d^2 - 14d + 1}{16d} \\
        \Rightarrow&\ y=\frac{\pm \sqrt{d(d^2-22d-7)}+\sqrt{d}^3-3\sqrt{d}}{4(d+1)}
    \end{align*}
    Therefore we get an upper-bound on the size outside the forbidden ball to be
    \begin{align*}
        &\left(\frac{- \sqrt{d(d^2-22d-7)}+\sqrt{d}^3-3\sqrt{d}}{4(d+1)}-0\right) 
         +\left(\frac{\sqrt{d}}{2}- \frac{ \sqrt{d(d^2-22d-7)}+\sqrt{d}^3-3\sqrt{d}}{4(d+1)}\right) \\
         &= \frac{\sqrt{d}}{2}-2 \frac{\sqrt{d(d^2-22d-7)}}{4(d+1)}=6\frac{1}{\sqrt{d}}+26\frac{1}{\sqrt{d}^3} + O\left(\frac{1}{\sqrt{d}^5} \right) = 6\frac{1}{\sqrt{d}}+ O\left(\frac{1}{\sqrt{d}^3} \right).
    \end{align*}

\end{proof}

\begin{figure}
    \centering
    \begin{tabular}{cc}
    \includegraphics[height = 0.3\textheight]{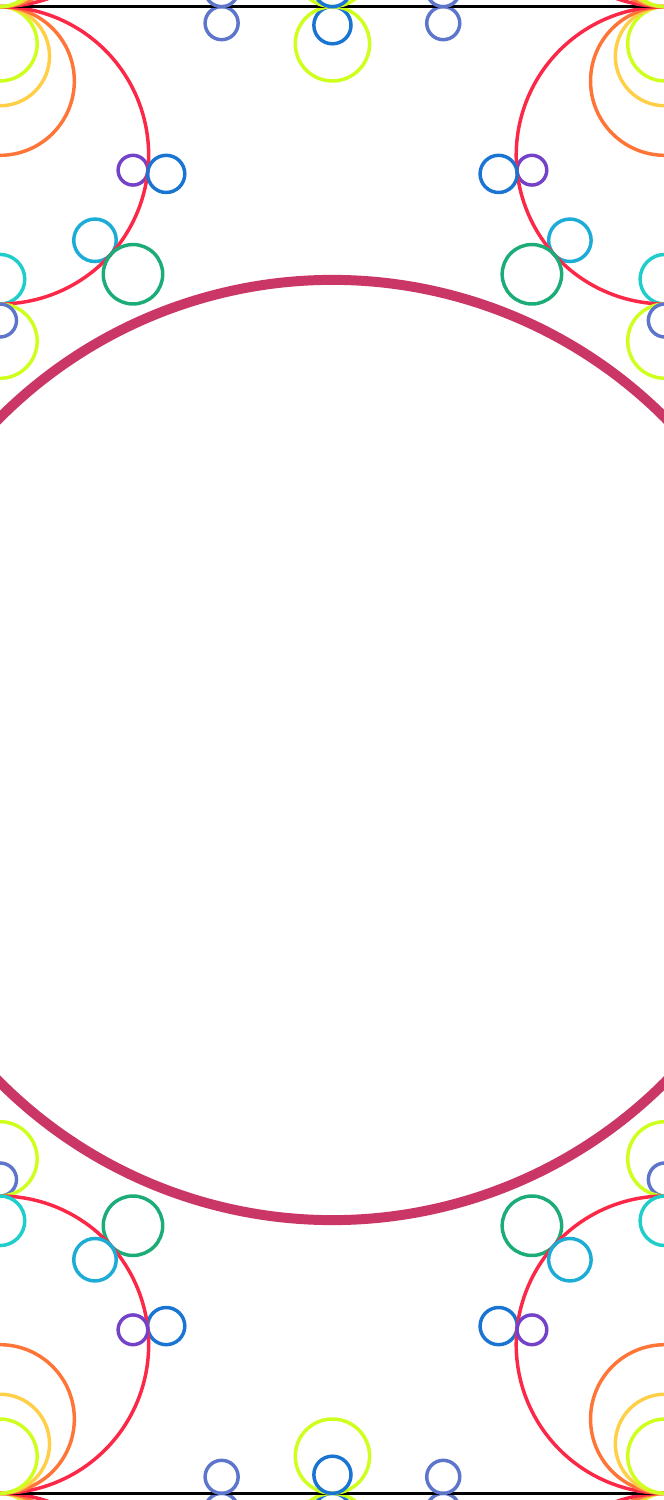} & \includegraphics[height = 0.3\textheight]{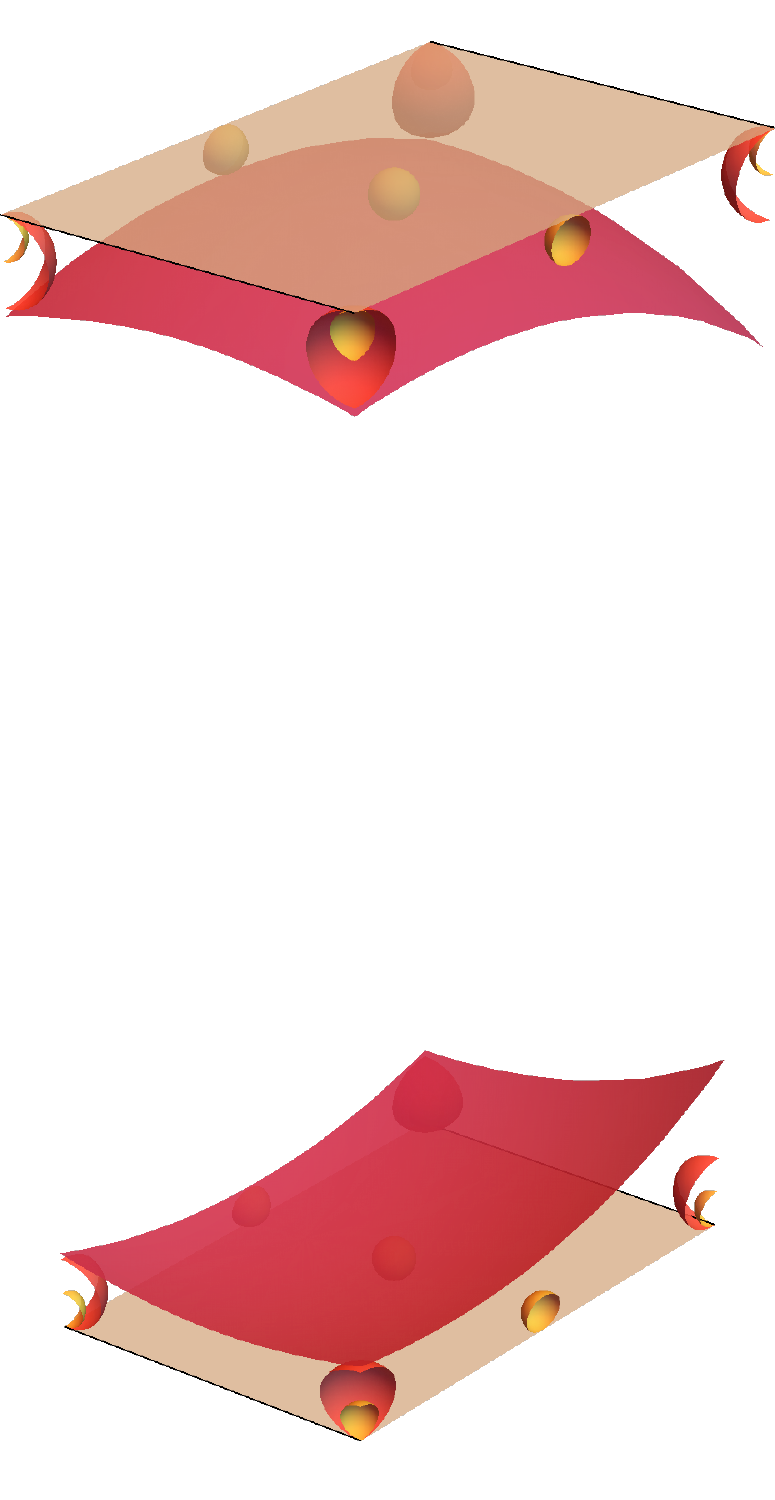}
    \end{tabular}
    \caption{A forbidden ball for $\ZZ[\sqrt{-5}]$ and for $\ZZ \oplus \ZZ \frac{1 + i}{2} \oplus \ZZ j \oplus \ZZ \frac{7j + ij}{14} \subset \left(\frac{-7,-14}{\QQ}\right)$.}
    \label{fig:forbidden_ball_examples}
\end{figure}

\begin{theorem}\label{thm:ghost_spheres_old}
For $\text{dim} = 4$, suppose that $R(\ZZ)$ has a normalized covering vector $u$ with $\nrm(u)>3$. Then one of the following is true:
    \begin{enumerate}
        \item $R(\ZZ)$ is $\ddagger$-Euclidean.
        \item $(R(\ZZ),u)$ is equivalent to one of orders listed in either Table \ref{tab:old_forbidden_balls} or \ref{tab:new_forbidden_balls} with $j$ as covering vector and has the associated forbidden ball for $\mathcal{S}_{R(\ZZ),u}$.
    \end{enumerate}
\end{theorem}
 
\begin{table}
    \centering
    \begin{align*}
    \begin{array}{c|l|l}
         R(\QQ) & R(\ZZ)^+ & \text{inv}_u(G) \\ \hline
         \left(\frac{-1,-1-4n}{\QQ}\right)& \ZZ \oplus \ZZ i \oplus \ZZ j & \sqrt{\frac{D}{D - 4}}\left(2,2,1,1,1\right) \\
         \left(\frac{-1,-2n}{\QQ}\right)& \ZZ \oplus \ZZ i \oplus \ZZ \frac{1 + i + j}{2} & \sqrt{\frac{D^2}{D^2 - 12D + 4}}\left(4,4,2,2,\frac{n-1}{n}\right) \\
         \left(\frac{-1,1 - 4n}{\QQ}\right)& \begin{cases} \ZZ \oplus \ZZ i \oplus \ZZ \frac{1 + j}{2} \\ \ZZ \oplus \ZZ i \oplus \ZZ \frac{i + j}{2} \end{cases} & \sqrt{\frac{D^2}{D^2 - 10D + 1}}\left(4,4,2,2,\frac{n-1}{n}\right) \\
         \left(\frac{-2,-1-4n}{\QQ}\right)& \ZZ \oplus \ZZ i \oplus \ZZ \frac{1 + i + j}{2} & \sqrt{\frac{D^2}{(D - 2)(D - 18)}}\left(4,4,2,2,\frac{4n - 2}{4n + 1}\right) \\
         \left(\frac{-2,1-4n}{\QQ}\right)& \ZZ \oplus \ZZ i \oplus \ZZ \frac{1 + j}{2} & \sqrt{\frac{D^2}{D^2-12D+4}}\left(4,4,2,2,\frac{4n - 2}{4n - 1}\right) \\
         \left(\frac{-2,-2(-3 + 8n)}{\QQ}\right)& \ZZ \oplus \ZZ i \oplus \ZZ \frac{2 + i + j}{4} & \sqrt{\frac{D^2}{D^2-18D+25}}\left(8,8,4,4,\frac{8n - 13}{8n - 3}\right) \\
         \left(\frac{-2,-2(-1 + 8n)}{\QQ}\right)& \ZZ \oplus \ZZ i \oplus \ZZ \frac{i + j}{4} & \sqrt{\frac{D^2}{D^2-14D+9}}\left(8,8,4,4,\frac{8n - 7}{8n - 1}\right) \\
         \left(\frac{-3,-n_3}{\QQ}\right)& \ZZ \oplus \ZZ \frac{1 + i}{2} \oplus \ZZ j & \sqrt{\frac{D}{D - 8}}\left(2,2,1,\frac{1}{3},1\right) \\
         \left(\frac{-3,-3(-1 + 3n)}{\QQ}\right)& \ZZ \oplus \ZZ \frac{1 + i}{2} \oplus \ZZ \frac{i + j}{3} & \sqrt{\frac{D}{D - 8}}\left(6,6,3,1,1\right) \\
         \left(\frac{-7,-n_7}{\QQ}\right)& \ZZ \oplus \ZZ \frac{1 + i}{2} \oplus \ZZ j & \sqrt{\frac{D}{D - 12}}\left(2,2,1,\frac{3}{7}  ,1\right) \\
         \left(\frac{-11,-11n_{11}}{\QQ}\right)& \ZZ \oplus \ZZ \frac{1 + i}{2} \oplus \ZZ j & \sqrt{\frac{D}{D - 8}}\left(2,2,1,\frac{5}{11},1\right)
    \end{array}
    \end{align*}
    \caption{Ghost spheres computed in \cite{Sheydvasser2019}---here, $D$ stands for $|\discrd(R(\ZZ))|$. The last three coordinates of each entry in the rightmost column should be understood as the $1,i,j$ components of $\xi_u(G)$. Any coefficients of $n$ are to be understood as positive integers; coefficients of the form $n_k$ are to be understood as positive integers not divisible by $k$.}
    \label{tab:old_forbidden_balls}
\end{table}

\begin{table}
    \centering
    \begin{align*}
    \begin{array}{c|l|l}
         R(\QQ) & R(\ZZ)^+ & \text{inv}_u(G) \\ \hline
         \left(\frac{-2,-2(1 + 8n)}{\QQ}\right)& \ZZ \oplus \ZZ i \oplus \ZZ \frac{i + j}{2} & \sqrt{\frac{D}{D^2 - 8D + 4}}\left(4,4,2,2,\frac{8n}{1 + 8n}\right) \\
         \left(\frac{-2,-2(3 + 8n)}{\QQ}\right)& \ZZ \oplus \ZZ i \oplus \ZZ \frac{i + j}{2} & \sqrt{\frac{D}{D^2 - 8D + 4}}\left(4,4,2,2,\frac{8n}{1 + 8n}\right) \\
         \left(\frac{-3,-3(1 + 3n)}{\QQ}\right)& \ZZ \oplus \ZZ \frac{1 + i}{2} \oplus \ZZ j & \sqrt{\frac{3D}{3D - 8}}\left(2,2,1,\frac{1}{3},1\right) \\
         \left(\frac{-7,-7n}{\QQ}\right)& \ZZ \oplus \ZZ \frac{1 + i}{2} \oplus \ZZ j & \sqrt{\frac{D}{D - 12}}\left(2,2,1,\frac{3}{7},1\right) \\
         \left(\frac{-11,-11n'}{\QQ}\right)& \ZZ \oplus \ZZ \frac{1 + i}{2} \oplus \ZZ j & \sqrt{\frac{11D}{11D - 8}}\left(2,2,1,\frac{5}{11},1\right) \\
         \left(\frac{-11,-11(2 + 11n)}{\QQ}\right)& \ZZ \oplus \ZZ \frac{1 + i}{2} \oplus \ZZ \frac{3i + j}{11} & \sqrt{\frac{D}{D-24}}\left(22,22,11,3,1\right) \\
         \left(\frac{-11,-11(-5 + 11n)}{\QQ}\right)& \ZZ \oplus \ZZ \frac{1 + i}{2} \oplus \ZZ \frac{4i + j}{11} & \sqrt{\frac{D}{D-28}}\left(22,22,0,4,1\right) \\
         \left(\frac{-11,-11(-4 + 11n)}{\QQ}\right)& \ZZ \oplus \ZZ \frac{1 + i}{2} \oplus \ZZ \frac{2i + j}{11} & \sqrt{\frac{D}{D-40}}\left(22,22,0,2,1\right) \\
         \left(\frac{-11,-11(-3 + 11n)}{\QQ}\right)& \ZZ \oplus \ZZ \frac{1 + i}{2} \oplus \ZZ \frac{5i + j}{11} & \sqrt{\frac{D}{D-8}}\left(22,22,0,-6,1\right) \\
         \left(\frac{-11,-11(-1 + 11n)}{\QQ}\right)& \ZZ \oplus \ZZ \frac{1 + i}{2} \oplus \ZZ \frac{i + j}{11} & \sqrt{\frac{D}{D-32}}\left(22,66,0,-10,1\right)
    \end{array}
    \end{align*}
    \caption{All $\ddagger$-orders for which new forbidden balls are constructed. Here, $D$ stands for $|\discrd(R(\ZZ))|$. The last three coordinates of each entry in the rightmost column should be understood as the $1,i,j$ components of $\xi_u(G)$. Any coefficients of $n$ are to be understood as non-negative integers and $n'$ is to be understood as a non-negative integer that is a square modulo $11$.}
    \label{tab:new_forbidden_balls}
\end{table}
 
\begin{proof}
Using a result of the second author \cite[Theorem 8.1]{Sheydvasser2019}, every such ring $R(\ZZ)$ is either $\ddagger$-Euclidean or is isomorphic to one of the orders listed in Table \ref{tab:old_forbidden_balls} or Table \ref{tab:new_forbidden_balls}. Furthermore, if $R(\ZZ)$ is not isomorphic to one of the orders listed in Table \ref{tab:new_forbidden_balls} and is not $\ddagger$-Euclidean, then there is a ghost sphere $G$ listed in Table \ref{tab:old_forbidden_balls} with the property that it does not intersect any sphere in $G(\ZZ).S_u$ \cite[Lemma 10.1]{Sheydvasser2019}. Therefore, $G$ is a forbidden ball.

Thus, it remains to find forbidden balls for the orders in Table \ref{tab:new_forbidden_balls}. The process is essentially the same: we will simply take an oriented sphere $G$ that is orthogonal to unit balls centered at points in $R(\ZZ)^+$. The proofs for the various cases are essentially the same, so we will just present the case where
    \begin{align*}
        \begin{array}{ll}
            R(\QQ) = \left(\frac{-11,-11(-1 + 11n)}{\QQ}\right) & R(\ZZ) = \ZZ \oplus \ZZ \frac{1 + i}{2} \oplus \ZZ \frac{i + j}{11} \oplus \ZZ \frac{11j + ij}{22}.
        \end{array}
    \end{align*}
    
\noindent To start, we know that if $\gamma \in \mathcal{W}$ and $u = j$ is a covering vector, then $\text{inv}_u(\gamma.S_u) = (0,0,u) \mod \discrd(R(\ZZ))$. In this case, $|\discrd(R(\ZZ))| = 11n - 1$. However, this can be sharpened a little, because it is easy to check by explicit computation in $R(\ZZ/2\ZZ)$ that actually $\text{inv}_u(\gamma.S_u) = (0,0,u) \mod 2$. This allows us to conclude that $\text{inv}_u(\gamma.S_u) = (0,0,u) \mod 2(11n - 1)$, so we may write it as
    \begin{align*}
        2(11n - 1)\left(k,k',a + b\frac{1 + i}{2} + c\frac{i + j}{11}\right) + (0,0,j)
    \end{align*}
    
\noindent for some $k,k',a,b,c \in \ZZ$. From this, it follows that
    \begin{align*}
        b_j&\left((22,66,-10i + j),\text{inv}_j(\gamma.S_j)\right) \\
        &= \begin{pmatrix} 22 \\ 66 \\ 0 \\ -10 \\ 1 \end{pmatrix}^T\begin{pmatrix} 0 & -\frac{1}{2} & 0 & 0 & 0 \\ -\frac{1}{2} & 0 & 0 & 0 & 0 \\ 0 & 0 & 1 & 0 & 0 \\ 0 & 0 & 0 & 11 & 0 \\ 0 & 0 & 0 & 0 & 11(-1 + 11n) \end{pmatrix} \left(2(11n - 1)\begin{pmatrix} k \\ k' \\ a + \frac{b}{2} \\ \frac{b}{2} + \frac{c}{11} \\ \frac{c}{11} \end{pmatrix} + \begin{pmatrix} 0 \\ 0 \\ 0 \\ 0 \\ 1 \end{pmatrix} \right) \\
        &= 11(-1 + 11n) - 22(-1 + 11n)(5b + c + k' + 3k - cn) \\
        &\in \nrm(j) + 2\nrm(j)\ZZ.
    \end{align*}
    
\noindent But this means that if we take $G$ to be the oriented sphere such that
    \begin{align*}
        \text{inv}_j(G) = \sqrt{11n - 1}{11n - 9}\left(2,2,1 + \frac{5}{11}i + j\right),
    \end{align*}
    
\noindent then for all $\gamma \in \mathcal{W}$,
    \begin{align*}
        b_j\left(\text{inv}_j(G),\text{inv}_j(\gamma.S_j\right) &\in \sqrt{11n - 1}{11n - 9}\left(\nrm(j) + 2\nrm(j)\ZZ\right),
    \end{align*}
    
\noindent and it is easy to see that this set does not intersect $[-\nrm(j),\nrm(j)]$, implying that $G$ is a forbidden ball.
\end{proof}

\begin{theorem}\label{thm: dim_5 case also has forbidden balls}
For $\text{dim} = 5$, suppose that $R(\ZZ)$ has a normalized covering vector $u$ with $\nrm(u) > 3$. Then one of the following is true:
    \begin{enumerate}
        \item $R(\ZZ)$ is $\ddagger$-Euclidean.
        \item $(R(\ZZ),u)$ is equivalent to one of the two orders listed in Table \ref{tab:forbidden_balls_dim5} with $ij/\gcd(i^2,j^2)$ as covering vector and has the associated forbidden ball for $\mathcal{S}_{R(\ZZ),u}$.
    \end{enumerate}
\end{theorem}

\begin{table}
    \centering
    \begin{align*}
    \begin{array}{c|l|l}
    R(\QQ) & R(\ZZ) & \text{inv}_u(G) \\ \hline
    \left(\frac{-1,-7}{\QQ}\right) & \ZZ \oplus \ZZ i \oplus \ZZ \frac{i + j}{2} \oplus \ZZ \frac{1 + ij}{2} & \left(7,7,\frac{7 + 7i + 3j + 3ij}{2}\right) \\
	\left(\frac{-2,-26}{\QQ}\right) & \ZZ \oplus \ZZ i \oplus \ZZ \frac{2 + i + j}{4} \oplus \ZZ \frac{2 + 2i + ij}{4} & \frac{1}{\sqrt{5}}\left(13,13,\frac{13 + 13i}{2} + j + \frac{5}{4}ij\right)
    \end{array}
    \end{align*}
    \caption{Forbidden balls corresponding to two orders of definite, rational quaternion algebras.}
    \label{tab:forbidden_balls_dim5}
\end{table}

\begin{proof}
Looking at Table \ref{tab:dim5_covering_vectors}, there are only two maximal orders that have normalized covering vector $u$ with $\nrm(u) > 3$, and those are the two listed in Table \ref{tab:forbidden_balls_dim5}. All that remains is calculating forbidden balls for them. The method is the same as the $dim = 4$ case---we simply take the $3$-sphere that is orthogonal to unit balls centered at points in $R(\ZZ)$. For brevity, we shall look at only the case where
    \begin{align*}
        R(\ZZ) = \ZZ \oplus \ZZ i \oplus \ZZ \frac{2 + i + j}{4} \oplus \ZZ \frac{2 + 2i + ij}{4} \subset \left(\frac{-2,-26}{\QQ}\right)
    \end{align*}
    
\noindent with $u = ij/2$; the other case is essentially the same, but simpler. We know that $\text{inv}_u(\gamma.S_u) = (0,0,u) \mod \discrd(R(\ZZ))$. Here, $|\discrd(R(\ZZ))| = \nrm(u) = 13$, so we may write
    \begin{align*}
        \text{inv}_u(\gamma.S_u) = 13\left(k,k',a + bi + c\frac{2 + i + j}{4} + d\frac{2+2i+ij}{4}\right) + \left(0,0,\frac{ij}{2}\right)
    \end{align*}
    
\noindent for some $k,k',a,b,c,d \in \ZZ$. From this, we get that
    \begin{align*}
        b_u&\left(\frac{1}{\sqrt{5}}\left(13,13,\frac{13 + 13i}{2} + j + \frac{5}{4}ij\right),\text{inv}_u(\gamma.S_u)\right) \\
        &= \frac{13}{2\sqrt{5}}\left(5 - 13k-13k'+52d+13a+26b+26c\right) \\
        &\in \frac{13}{2\sqrt{5}}\left(5 + 13\ZZ\right).
    \end{align*}
    
\noindent It isn't hard to check that this set does not intersect $[-1,1]$, hence the $G$ listed in Table \ref{tab:forbidden_balls_dim5} is a forbidden ball.
\end{proof}

The construction of these forbidden balls finally allows us to prove our main theorems.

\begin{proof}[Proof of Theorem \ref{thm: main theorem}]
Suppose that $R(\ZZ)$ has a covering vector $u$; we can assume without loss of generality that $u$ is normalized. Then if $|\discrd(R(\ZZ))| > 33$, it follows that $\nrm(u) > 3$ by Theorem \ref{thm: covering vectors are excellent}. By Theorems \ref{thm:n=3forbiddenball}, \ref{thm:ghost_spheres_old}, and \ref{thm: dim_5 case also has forbidden balls}, we know that $\mathcal{S}_{R(\ZZ),u}$ admits a forbidden ball. It follows immediately from Theorem \ref{thm: forbidden balls force infinite index} that $E(\ZZ)$ is infinite index in $G(\ZZ)$. Since the forbidden ball $B$ can be translated by $R(\ZZ)^+$ so that its image is inside the fundamental parallelogram and $\mathcal{A}_{R(\ZZ),u}$ is contained inside $\mathcal{S}_{R(\ZZ),u}$, it is obvious that $\mathcal{A}_{R(\ZZ),u}$ has density less than $1$. Finally, the two planes in $\mathcal{A}_{R(\ZZ),u}$ that bound the fundamental parallelogram must be separated by a distance of $\sqrt{\nrm(u)}$. We know that the forbidden ball is in each case the unique oriented sphere orthogonal to unit balls centered at the vertices of this fundamental parallelogram---in the limit, the proportion of the parallelogram left after cutting out by this sphere is then the same as if we took the volume outside a pair of infinite planes tangent to the unit spheres centered at the vertices at the top and bottom of the fundamental parallelogram. This proportion is
    \begin{align*}
        \frac{2}{\sqrt{\nrm(u)}} \xrightarrow{|\disc(R(\ZZ))| \rightarrow \infty} 0,
    \end{align*}
    
\noindent so the density does go to $0$. A more careful analysis is given in Corollary \ref{cor:n=3forbiddensity}.
\end{proof}

\begin{proof}[Proof of Theorem \ref{thm: unreasonable slightness for circle packings}]
This is mostly just a special case of Theorem \ref{thm: main theorem}, with two main additions: for $dim = 3$ case, $\nrm(u) > 3$ is guaranteed as long as $K \neq \QQ(\sqrt{-1}), \QQ(\sqrt{-2})$, or $\QQ(\sqrt{-3})$, all of which are norm-Euclidean---thus, our results are enough to conclude that $E(2,\mathfrak{o}_K$ is infinite-index in $SL(2,\mathfrak{o}_K)$ and $\delta(\mathcal{A}_K) \rightarrow 0$. We have no way, to our knowledge, of utilizing our machinery to prove that the group $E(2,\mathfrak{o}_K$ is non-normal, but this is already known \cite{Nica_2011}.
\end{proof}

\bibliography{AlgebraicIsomorphism}
\bibliographystyle{plain}
\end{document}